\documentclass[11pt]{article}
\usepackage[margin=1in]{geometry}
\usepackage{amsmath,amssymb,amsthm}
\usepackage{graphicx}
\usepackage{subfig}
\usepackage{float}
\usepackage{hyperref}
\graphicspath{{paper_figures_separate/}}
\newcommand{\email}[1]{\href{mailto:#1}{\nolinkurl{#1}}}

\usepackage{bm}
\usepackage{xcolor}
\usepackage{algorithm}
\usepackage{algpseudocode}
\usepackage{enumitem}
\theoremstyle{plain}
\newtheorem{theorem}{Theorem}[section]
\newtheorem{proposition}{Proposition}[section]
\newtheorem{lemma}{Lemma}[section]
\newtheorem{corollary}{Corollary}[section]
\newtheorem{problem}{Problem}[section]
\newtheorem{Remark}{Remark}[section]
\newtheorem{example}{Example}[section]

\newcommand{\x}{{\bf x}}

\title{Carleman--Picard and time-dimensional reduction for inverse initial-data problems in nonlinear transport with memory
}

\author{%
\begin{minipage}{0.9\textwidth}
\raggedright
{\large\bfseries Navaraj Neupane}\\
{\small
Department of Mathematics and Statistics, University of North Carolina
at Charlotte, Charlotte, NC 28223, USA,\\
Email: \email{nneupan2@charlotte.edu} (corresponding author)}\\[1.2em]
{\large\bfseries Loc Hoang Nguyen}\\
{\small
Department of Mathematics and Statistics, University of North Carolina
at Charlotte, Charlotte, NC 28223, USA,\\
Email: \email{loc.nguyen@charlotte.edu}}
\end{minipage}}
\date{}
\numberwithin{equation}{section}
\begin{document}
\maketitle

\begin{abstract}
We study an inverse initial-data problem for a quasilinear transport
equation with nonlinear and memory effects. The unknown initial state is
reconstructed from time-dependent measurements on the outflow boundary,
with prescribed inflow data. We first apply a Legendre--exponential
time-dimensional reduction to transform the governing equation into a
finite nonlinear system in space. We then develop a
Carleman-weighted and Tikhonov-regularized Picard method. At each
iteration, the nonlinear terms are evaluated using the previous
iterate, leading to a linear minimization problem with a unique
solution. A Carleman estimate for the principal transport operator is
used to prove that, for a sufficiently large Carleman parameter, the
resulting Picard map is contractive on a prescribed admissible set.
Consequently, the method converges from an arbitrary initial guess in
that set. We also establish stability with respect to noisy outflow
data. These analytical results concern the truncated and regularized
reduced problem. Two-dimensional numerical experiments demonstrate
accurate reconstruction of single and multiple inclusions, robustness
with respect to noise, and rapid convergence of the Picard iteration.
\end{abstract}



\section{Introduction}
\label{sec:intro}

Let $\Omega\subset\mathbb{R}^d$, $d\geq 1$, be a bounded domain with
smooth boundary $\partial\Omega$, and let $T>0$ be fixed. We set
\[
Q_T:=\Omega\times(0,T) \qquad 
\mbox{and}
\qquad \Gamma_T = \partial \Omega \times (0, T),
\]
and define the inflow and outflow portions of the boundary by
\[
\Gamma_-:=\left\{\x\in\partial\Omega:
H(\x)\cdot\nu(\x)\leq 0\right\},
\qquad
\Gamma_+:=\left\{\x\in\partial\Omega:
H(\x)\cdot\nu(\x)>0\right\},
\]
where $\nu(\x)$ denotes the outward unit normal vector to
$\partial\Omega$ at $\x$.

We consider the following nonlinear first-order transport equation with
a memory term:
\begin{equation}
\begin{cases}
\displaystyle
\partial_t u
+c(\x,u)H(\x)\cdot\nabla u
+f(\x,u)
+\int_0^t \alpha(\x,t-t')u(\x,t')\,dt'=0,
& (\x,t)\in Q_T,\\[6pt]
u(\x,0)=u^0(\x),
& \x\in\Omega,\\[4pt]
u(\x,t)=g(\x,t),
& (\x,t)\in\Gamma_-\times(0,T).
\end{cases}
\label{eq:pde}
\end{equation}
Here, $u=u(\x,t)$ is the scalar state variable,
$H:\overline{\Omega}\to\mathbb{R}^d$ is a prescribed advection field,
$c:\overline{\Omega}\times\mathbb{R}\to\mathbb{R}$ is a
state-dependent transport coefficient. We assume that there exists a
constant $\kappa_0>0$ such that
\begin{equation}
  c(\x,s)\geq\kappa_0
  \qquad
  \text{for all }\x\in\overline\Omega
  \text{ and }s\in\mathbb R.
  \label{eq:nondegenerate-velocity}
\end{equation}
$f:\overline{\Omega}\times\mathbb{R}\to\mathbb{R}$ is a nonlinear
lower-order term, and
$
\alpha:\overline{\Omega}\times\mathbb{R}\to\mathbb{R}
$
is a given memory kernel.

Physically, the vector field $H(\x)$ determines the local transport
direction and spatial flow pattern, whereas the coefficient $c(\x,u)$
scales the magnitude of the transport velocity. The effective transport
velocity in \eqref{eq:pde} is $c(\x,u)H(\x)$. In particular, the
dependence of $c$ on $u$ allows the state itself to influence its
propagation speed along the prescribed direction $H(\x)$.
At the boundary, the condition $H(\x)\cdot\nu(\x)<0$ indicates that the
flow enters $\Omega$, whereas $H(\x)\cdot\nu(\x)>0$ indicates that it
leaves $\Omega$. Accordingly, the boundary data are prescribed on the
inflow boundary $\Gamma_-$.
The convolution-type term
$
\int_0^t \alpha(\x,t-t')u(\x,t')\,dt'
$
incorporates memory into the model by allowing the state at time $t$ to
depend on its history over the interval $[0,t]$. The kernel $\alpha$
determines how strongly past states influence the current dynamics.

We consider the following inverse initial-value problem.

\begin{problem}[Inverse initial-data problem]
Assume that $u\in C^1(\overline{\Omega}\times[0,T]) \cap H^3((0, T); H^p(\Omega))$. Given the
outflow boundary observation
\[
h(\x,t):=u(\x,t),
\qquad
(\x,t)\in\Gamma_+\times(0,T),
\]
reconstruct the initial state $u^0(\x)$ for all $\x\in\Omega$.
\label{p}
\end{problem}

Nonlinear transport equations and evolution equations with
history-dependent effects arise in several areas of applied
mathematics. Classical scalar conservation laws used in traffic-flow
modeling provide fundamental examples of state-dependent transport
\cite{LighthillWhitham1955,Richards1956}, while memory, retention,
delayed-response, and nonlocal-in-time effects occur in models of
solute transport in heterogeneous porous media
\cite{VanGenuchtenWagenet1989}, transport through fractured geological
formations \cite{GrafSimmons2009}, and rate-dependent granular flows
\cite{SavageHutter1989,PudasainiHutter2007}. Related mathematical
models include scalar conservation laws with fading memory, in which a
Volterra-type convolution acts on the flux
\cite{ChenChristoforou2007}, as well as measure transport equations
with memory effects \cite{CamilliDeMaio2019}.

Inverse problems for transport equations have also been studied in
both stationary and time-dependent settings. A substantial part of
this literature concerns inverse transport and radiative transfer,
where unknown coefficients or sources are reconstructed from boundary
or albedo-type measurements
\cite{Bal2009,BalJollivet2008,BalJollivet2018,BalTamasan2007}.
For linear time-dependent first-order transport equations, Carleman
estimates have been used to establish global uniqueness
\cite{KlibanovPamyatnykh2008} and Lipschitz stability
\cite{KlibanovPamyatnykh2006Stability} for coefficient inverse
problems. Related coefficient inverse problems have also been treated
using local Carleman estimates
\cite{CannarsaFloridiaGolgeleyenYamamoto2019,GaitanOuzzane2014}.
Parameter reconstruction has been studied for more general transport
models involving external fields, absorption, and source terms
\cite{LaiLi2020}, while inverse problems for nonlinear
time-dependent transport equations have recently been investigated
\cite{LaiZhou2024}. Another related direction is the identification
of memory kernels in evolution equations with history-dependent
effects, including strongly damped wave equations
\cite{Colombo2007,ColomboGuidetti2009} and viscoelastic models with
time- and space-dependent kernels \cite{JannoVonWolfersdorf2001}.

Although these models do not necessarily coincide with
\eqref{eq:pde}, they motivate the study of transport processes whose
present evolution depends on both the current state and its previous
history. In many practical situations, the initial state inside the
spatial domain is not directly accessible, whereas measurements can be
collected on the boundary over a later time interval. Recovering the
initial condition from such observations provides the starting state
needed to solve the forward problem and predict the subsequent
evolution of the system. Accordingly, \eqref{eq:pde} is treated as a
prototype quasilinear transport equation with a lower-order Volterra
memory term. The inverse problem considered here consists of
recovering its unknown initial state from outflow boundary
observations, with prescribed inflow data. To the best of our
knowledge, inverse initial-data reconstruction for quasilinear
transport equations combining state-dependent transport and Volterra
memory effects has received comparatively little attention.

Solving Problem~\ref{p} involves two principal difficulties. First, the
forward equation is nonlinear. This nonlinearity arises from the
dependence of the transport coefficient $c(\x,u)$ and the lower-order
term $f(\x,u)$ on the unknown state $u$. In addition, the Volterra term
couples the state at the current time with its entire previous history,
making the equation nonlocal in time. Consequently, the relation
between the unknown initial condition and the measured outflow data is
both nonlinear and history-dependent. Such memory and nonlocal-in-time effects arise in transport models
with retention, relaxation, or delayed-response mechanisms
\cite{ChenChristoforou2007,CamilliDeMaio2019}. By contrast, inverse
problems recovering the memory kernel itself have been studied for
other evolution models, including strongly damped wave equations and
viscoelastic systems \cite{Colombo2007,JannoVonWolfersdorf2001}.

Second, the reconstruction has a
backward character: one seeks to determine the state inside $\Omega$ at
the initial time from measurements collected later only on the
boundary. Such a reconstruction can be highly sensitive to measurement
noise and modeling errors. In the sense of Hadamard, backward inverse problems
are typically ill-posed, since small perturbations in the data may lead
to large errors in the reconstructed state unless appropriate
regularization is imposed \cite{TikhonovArsenin1977,LattesLions1969}.
The difficulty is further amplified by the fact that only the outflow
portion $\Gamma_+$ of the boundary is observed, rather than full
boundary or interior data. Thus an effective reconstruction method must
control both the instability caused by the backward nature of the
problem and the loss of information associated with partial boundary
observations.

To address these difficulties, we combine
regularization with a Carleman-based globally convergent technique.
Among the available approaches, including Carleman convexification \cite{KlibanovIoussoupova:SMA1995}, the
Carleman contraction principle \cite{LeNguyen2022, Nguyen2023CarlemanContraction}, and the Carleman--Newton method \cite{AbhishekLeNguyenKhan, LeNguyenTran:CAMWA2022}, we
adopt the Carleman contraction principle because it leads to a
computationally inexpensive iteration, converges rapidly in practice,
and is comparatively simple to implement. The regularization stabilizes
the reconstruction, while the Carleman estimate makes the resulting
Picard map contractive and thereby provides a systematic treatment of
the nonlinear reduced system. The
time-dimensional reduction technique itself has been used to recover
initial states for several other classes of evolution equations,
including anisotropic elastic systems, nonlinear parabolic and
hyperbolic equations, the nonlinear Schr\"odinger equation, Maxwell's
equations, the Navier--Stokes equations, and convection-diffusion
equations with memory \cite{DangLeLuuNguyen2025Elastic, DangNguyenVu2024, Le2023, LeNguyenNguyenPark2024, LeVanDangNguyen2025Maxwell, NeupaneNguyen2026NLS, NguyenNguyen2026NavierStokesForce, VanLeNguyen2026, VanNguyenTranNguyen2026Memory}.

The Carleman contraction principle can be viewed as a globalized version
of the classical Picard iteration. For a standard Picard scheme, the
convergence is generally local because the stability, or coercivity,
constant associated with the principal linear operator is fixed. Hence,
the nonlinear terms can be controlled only when the current iterate is
sufficiently close to the exact solution. In the Carleman contraction
principle, each frozen linearized problem is instead solved by
minimizing a Carleman-weighted regularized functional. The corresponding
Carleman estimate produces a coercive term whose strength increases
with the Carleman parameter $\lambda$, whereas the Lipschitz constant
of the nonlinear terms remains fixed on the admissible set. By choosing
$\lambda$ sufficiently large, the strengthened coercivity dominates
the nonlinear perturbations and makes the resulting Picard map
contractive. Consequently, the initial guess need not be close to the
exact solution; it may be chosen arbitrarily within the prescribed
admissible set.

The remainder of the paper is organized as follows.
Section~\ref{sec:preliminaries} states the geometric propagation
condition, establishes the Carleman estimate for $H\cdot\nabla$, and
recalls the Legendre--exponential basis. Section~\ref{sec:reduction}
derives the finite nonlinear system for the spatial modal coefficients
by time-dimensional reduction. Section~\ref{sec:contraction} introduces
the Carleman-weighted Picard map and proves its contraction property and
global convergence within the admissible set. Section~\ref{sec:stability}
establishes stability of the reduced reconstruction with respect to
noisy outflow data. Section~\ref{sec:numerics} describes the numerical
implementation and presents the reconstruction experiments, including
the study of the Carleman parameter. Finally,
Section~\ref{sec:conclusion} summarizes the main results.
\section{Preliminaries}
\label{sec:preliminaries}

This section collects the main ingredients needed for the subsequent
time-dimensional reduction and the convergence and stability analyses.
We first state the geometric propagation assumptions, then establish a
Carleman estimate for the principal transport operator, and finally
recall the Legendre--exponential basis used in the time-dimensional
reduction.

\subsection{Geometric propagation condition}

To recover the initial state $u^0$ throughout $\Omega$ from measurements
on $\Gamma_+$, the information transported from every point of
$\Omega$ must reach the observed outflow boundary within the time
interval $(0,T)$. We now formulate assumptions on $H$, $c$, and $T$
that ensure this propagation property.
 To state them
independently of the unknown solution, let
$Y(\sigma;\x)$ denote the forward flow generated by $H$:
\begin{equation}
  \frac{dY}{d\sigma}=H(Y),
  \qquad
  Y(0;\x)=\x.
  \label{eq:H-flow}
\end{equation}
Here, $\sigma$ is the flow parameter along the integral curves of $H$;
it is distinct from the physical time variable $t$ in
\eqref{eq:pde}.
For $\x\in\Omega$, define the forward outflow exit parameter
\begin{equation}
  \tau_H(\x)
  :=
  \inf\{\sigma>0:Y(\sigma;\x)\in\Gamma_+\}.
  \label{eq:H-exit-parameter}
\end{equation}
We assume that $H$ is nontrapping: every forward integral curve of
$H$ starting in $\Omega$ reaches the outflow boundary $\Gamma_+$ in
finite time. More precisely, we impose the uniform nontrapping
condition
\begin{equation}
  \tau_H(\x)<\infty\quad\text{for every }\x\in\Omega,
  \qquad
  \tau_H^*
  :=
  \sup_{\x\in\Omega}\tau_H(\x)<\infty.
  \label{eq:nontrapping}
\end{equation}
We then require the observation time to satisfy
\begin{equation}
  T>\frac{\tau_H^*}{\kappa_0}.
  \label{eq:full-outflow-observability}
\end{equation}
Because the physical transport velocity is $c(\x,u)H(\x)$ and
$c\geq\kappa_0$, its travel time to $\Gamma_+$ is at most
$\tau_H(\x)/\kappa_0<T$. Consequently, the observation interval is
long enough for the entire initial state to propagate to the observed
boundary. 

The preceding conditions encode a necessary geometric observability
requirement for uniqueness: no part of the initial state can remain
invisible to the outflow measurement throughout $(0,T)$. Conditions
\eqref{eq:nontrapping} and
\eqref{eq:full-outflow-observability} provide a convenient sufficient
criterion for this propagation requirement, but they should not be
interpreted as a complete set of sufficient conditions for uniqueness
of the original inverse problem. Establishing sharp conditions and a
rigorous uniqueness theorem for the quasilinear transport equation
with a state-dependent velocity and a nonlocal memory term is a
challenging problem and lies outside the scope of this paper.  Accordingly, the solvability, uniqueness, convergence, and stability
results established in this paper are understood within the
approximation framework of the time-dimensional reduction. In other
words, these results apply to the truncated system for the temporal
coefficients, which serves as an approximation of the original inverse
problem for $u(\x,t)$.

\subsection{Carleman estimate}
\label{sec:carleman}

The main analytical tool used in the subsequent convergence and
stability analysis is a Carleman estimate for the principal transport
operator
\[
  u\mapsto H(\x)\cdot\nabla u.
\]
This is the principal part of the reduced system
\eqref{eq:reduced-full-system}. The lower-order terms will be treated
after the estimate for this operator has been established.

Throughout this section, we assume that
\[
  H\in C^1(\overline{\Omega};\mathbb{R}^d).
\]
We assume that there exists a function
\[
  \varphi_*\in C^2(\overline{\Omega})
\]
such that
\begin{equation}
  \label{eq:carleman-weight-condition}
  H(\x)\cdot\nabla\varphi_*(\x)\geq \mu_0>0,
  \qquad \x\in\overline{\Omega}
\end{equation}
Thus $\varphi_*$ is increasing along the integral curves of $H$.
Under appropriate geometric assumptions on
$H$ and $\Omega$, such weights can be constructed; see, for
example,~\cite{CannarsaFloridiaGolgeleyenYamamoto2019}. Here we assume the
existence of a global $C^2$ weight in order to keep the presentation
focused on the inverse reconstruction method.

\begin{example}\label{ex:carleman-weight}
Let $d=2$, $\Omega=(-1,1)^2$, and
\[
  H(x,y)
  =
  \left(
  1+0.25\sin(\pi y)\cos(\pi x),
  0.5+0.20\cos(\pi x)\cos(\pi y)
  \right)^\top.
\]
We choose
\[
  \varphi_*(x,y)=x+0.5y.
\]
Then
\[
  H\cdot\nabla\varphi_*
  =
  1.25
  +0.25\sin(\pi y)\cos(\pi x)
  +0.10\cos(\pi x)\cos(\pi y)
  \geq 0.9.
\]
Therefore, \eqref{eq:carleman-weight-condition} holds with
$\mu_0=0.9$.
\end{example}

The following Carleman estimate involves only the principal operator
$H(\x)\cdot\nabla$.

\begin{proposition}[Carleman estimate]
  \label{prop:transport-carleman}
  Let $\varphi_*$ satisfy \eqref{eq:carleman-weight-condition}. Then
  there exist constants $\lambda_0>0$ and $C>0$, depending only on
  $\Omega$, $H$, $\mu_0$, and $\varphi_*$, such that, for all
  $\lambda\geq\lambda_0$ and all $v\in H^1(\Omega)$,
  \begin{equation}
    \lambda\int_{\Omega}e^{2\lambda\varphi_*}|v|^2\,d\x
    \leq
    C\int_{\Omega}e^{2\lambda\varphi_*}
    |H\cdot\nabla v|^2\,d\x
    +
    C\lambda\int_{\partial\Omega}
    e^{2\lambda\varphi_*}|v|^2\,dS.
    \label{eq:transport-carleman}
  \end{equation}
\end{proposition}

\begin{proof}
Set $w=e^{\lambda\varphi_*}v$. Then
\[
  e^{\lambda\varphi_*}H\cdot\nabla v
  =
  H\cdot\nabla w
  -\lambda(H\cdot\nabla\varphi_*)w.
\]
Multiplying this identity by $-w$ and integrating over $\Omega$, we
first obtain
\[
  -\int_\Omega
  e^{\lambda\varphi_*}(H\cdot\nabla v)w\,d\x
  =
  -\int_\Omega(H\cdot\nabla w)w\,d\x
  +\lambda\int_\Omega
  (H\cdot\nabla\varphi_*)|w|^2\,d\x.
\]
Since
\[
  (H\cdot\nabla w)w
  =
  \frac12H\cdot\nabla(|w|^2)
\]
and
\[
  \operatorname{div}(H|w|^2)
  =
  (\operatorname{div}H)|w|^2
  +H\cdot\nabla(|w|^2),
\]
the divergence theorem gives
\[
  -\int_\Omega(H\cdot\nabla w)w\,d\x
  =
  -\frac12\int_{\partial\Omega}
  (H\cdot\nu)|w|^2\,dS
  +\frac12\int_\Omega(\operatorname{div}H)|w|^2\,d\x.
\]
Consequently,
\[
  -\int_\Omega
  e^{\lambda\varphi_*}(H\cdot\nabla v)w\,d\x
  =
  -\frac12\int_{\partial\Omega}
  (H\cdot\nu)|w|^2\,dS
  +\frac12\int_\Omega(\operatorname{div}H)|w|^2\,d\x
  +\lambda\int_\Omega
  (H\cdot\nabla\varphi_*)|w|^2\,d\x.
\]
Using \eqref{eq:carleman-weight-condition} and estimating the boundary
term by its absolute value gives
\begin{equation*}
  -\int_\Omega
  e^{\lambda\varphi_*}(H\cdot\nabla v)w\,d\x
  \geq
  \left(\lambda\mu_0
  -\frac12\|\operatorname{div}H\|_{L^\infty(\Omega)}\right)
  \int_\Omega|w|^2\,d\x
  -\frac12\int_{\partial\Omega}|w|^2|H\cdot\nu|\,dS.
\end{equation*}
On the other hand, Cauchy's inequality gives
\[
  \left|
  \int_\Omega
  e^{\lambda\varphi_*}(H\cdot\nabla v)w\,d\x
  \right|
  \leq
  \frac{1}{\lambda\mu_0}
  \int_\Omega e^{2\lambda\varphi_*}
  |H\cdot\nabla v|^2\,d\x
  +
  \frac{\lambda\mu_0}{4}
  \int_\Omega|w|^2\,d\x.
\]
Combining the last two inequalities, we obtain
\[
  \Big(
  \frac{3\lambda\mu_0}{4}
  -\frac12\|\operatorname{div}H\|_{L^\infty(\Omega)}
  \Big)
  \int_\Omega|w|^2\,d\x
  \leq
  \frac{1}{\lambda\mu_0}
  \int_\Omega e^{2\lambda\varphi_*}
  |H\cdot\nabla v|^2\,d\x
  +
  \frac12\int_{\partial\Omega}|w|^2|H\cdot\nu|\,dS.
\]
Choose $\lambda_0\geq1$ sufficiently large that the coefficient on the
left-hand side is bounded below by $\lambda\mu_0/4$ for all
$\lambda\geq\lambda_0$. Substituting
$w=e^{\lambda\varphi_*}v$, using $\lambda\geq1$, and absorbing
$\|H\cdot\nu\|_{L^\infty(\partial\Omega)}$ into $C$ yields
\eqref{eq:transport-carleman}.
\end{proof}

Applying Proposition~\ref{prop:transport-carleman} componentwise gives the
corresponding estimate for vector-valued functions.

\begin{corollary}
  \label{cor:system-carleman}
  Let $N\geq0$ and assume that
  \eqref{eq:carleman-weight-condition} holds. Then there exist constants
  $\lambda_0>0$ and $C>0$, depending only on $\Omega$, $H$, $\mu_0$,
  and $\varphi_*$, such that, for all $\lambda\geq\lambda_0$ and all
  $U=(u_0,\ldots,u_N)^\top\in[H^1(\Omega)]^{N+1}$,
  \begin{equation}
    \lambda\int_\Omega e^{2\lambda\varphi_*}|U|^2\,d\x
    \leq
    C\int_\Omega e^{2\lambda\varphi_*}
    |H\cdot\nabla U|^2\,d\x
    +
    C\lambda\int_{\partial\Omega}
    e^{2\lambda\varphi_*}|U|^2\,dS,
    \label{eq:system-carleman}
  \end{equation}
  where $|U|^2=\displaystyle\sum_{m=0}^{N}|u_m|^2$.
\end{corollary}

\subsection{Legendre--exponential basis}
\label{sec:legendre-exponential-basis}

We next recall the Legendre--exponential basis used in the
time-dimensional reduction. This basis is particularly convenient
because its differentiation properties allow the time derivative in
\eqref{eq:pde} to be represented in terms of the spatial modal
coefficients.
 Let $P_n$ denote the Legendre
polynomial of degree $n$ on $[-1,1]$, normalized by
\[
  \int_{-1}^{1}P_n(s)P_m(s)\,ds
  =
  \frac{2}{2n+1}\delta_{nm}.
\]
For $t\in[0,T]$, define
\begin{equation}
  \Psi_n(t)
  :=
  \sqrt{\frac{2n+1}{T}}\,
  e^t P_n\left(\frac{2t}{T}-1\right),
  \qquad n\geq0.
  \label{eq:legendre-exponential-basis}
\end{equation}
This basis was introduced in
\cite{DangLeLuuNguyen2025Elastic}.
The family $\{\Psi_n\}_{n\geq0}$ is an orthonormal basis of the
weighted space $L^2_{e^{-2t}}(0,T)$ with inner product
\begin{equation}
  \label{eq:inner-product}
  \langle f,g\rangle_{e^{-2t}}
  :=
  \int_0^T e^{-2t}f(t)g(t)\,dt.
\end{equation}

For
$u\in L^2_{e^{-2t}}((0,T);H^p(\Omega))$, $p\geq0$, define
\begin{equation}
  u_n(\x)
  :=
  \langle u(\x,\cdot),\Psi_n\rangle_{e^{-2t}}
  =
  \int_0^T e^{-2t}u(\x,t)\Psi_n(t)\,dt.
  \label{eq:temporal-coefficients}
\end{equation}

We recall the differentiation result established for this basis in
\cite[Theorem~1]{DangLeLuuNguyen2025Elastic}.

\begin{proposition}[Differentiation of the time expansion]
\label{prop:termwise-differentiation}
Let $p\geq0$, let
$u\in L^2_{e^{-2t}}((0,T);H^p(\Omega))$, and let $u_n$ be given by
\eqref{eq:temporal-coefficients}. For $j\in\{1,2\}$, suppose that
\begin{equation}
  \sum_{n=0}^{\infty}u_n(\x)\Psi_n^{(j)}(t)
  \label{eq:differentiated-infinite-series}
\end{equation}
converges in
$L^2_{e^{-2t}}((0,T);H^p(\Omega))$. Then the weak derivative
$\partial_t^j u$ belongs to this space and
\begin{equation}
  \partial_t^j u(\x,t)
  =
  \sum_{n=0}^{\infty}u_n(\x)\Psi_n^{(j)}(t)
  \quad\text{in }
  L^2_{e^{-2t}}((0,T);H^p(\Omega)).
  \label{eq:commutation-time-series}
\end{equation}
Moreover, the convergence assumption is satisfied if
\[
  u\in H^k((0,T);H^p(\Omega))
  \qquad\text{with}\qquad
  k\geq2j+1.
\]
In particular, $k\geq3$ suffices for the first derivative used in this
paper, while $k\geq5$ suffices for the second derivative.
\end{proposition}

\begin{Remark}
\label{rem:choice-of-time-basis}
The Legendre--exponential basis is used because none of its basis
functions has an identically vanishing derivative \cite[Proposition 2.1]{DangLeLuuNguyen2025Elastic}:
\[
  \Psi_n'\not\equiv 0
  \qquad \text{for every } n\geq 0.
\]
Indeed, under the truncated expansion
\[
  u(\x,t)\approx\sum_{n=0}^N u_n(\x)\Psi_n(t),
\]
the time derivative is approximated by
\[
  u_t(\x,t)\approx\sum_{n=0}^N u_n(\x)\Psi_n'(t).
\]
Hence, no retained spatial coefficient $u_n$ is eliminated
identically when the truncated temporal expansion is differentiated. This property is not shared by many
standard orthonormal bases. For example, a polynomial basis typically
contains a constant first function whose derivative vanishes, causing
the corresponding modal coefficient to disappear from the
time-derivative term. The exponential factor in the present basis
prevents this degeneracy. In addition, truncating the temporal
expansion serves as a regularization mechanism by suppressing
high-order temporal modes that may be sensitive to noise.
\end{Remark}

\section{Time-dimensional reduction}
\label{sec:reduction}

We now begin the time-dimensional reduction. The first step is to
divide both sides of \eqref{eq:pde} by $c(\x,u)$. This division is well
defined because of \eqref{eq:nondegenerate-velocity}. We obtain
\begin{equation}
  \frac{1}{c(\x,u)}\partial_tu
  +H(\x)\cdot\nabla u
  +\frac{f(\x,u)}{c(\x,u)}
  +\frac{1}{c(\x,u)}
   \int_0^t\alpha(\x,t-t')u(\x,t')\,dt'
  =0
  \qquad\text{in }Q_T.
  \label{eq:revised_pde}
\end{equation}

By the completeness of the basis and
Proposition~\ref{prop:termwise-differentiation}, we expand $u$ and its
time derivative as
\begin{equation}
  u(\x,t)
  =
  \sum_{n=0}^{\infty}u_n(\x)\Psi_n(t),
  \qquad
  \partial_tu(\x,t)
  =
  \sum_{n=0}^{\infty}u_n(\x)\Psi_n'(t).
  \label{eq:infinite-time-expansions}
\end{equation}
These equalities are understood in the weighted spaces specified in
Proposition~\ref{prop:termwise-differentiation}.

Substituting \eqref{eq:infinite-time-expansions} into
\eqref{eq:revised_pde} gives
\begin{multline}
  \frac{1}{
    \displaystyle
    c\left(\x,\sum_{n=0}^{\infty}u_n(\x)\Psi_n(t)\right)}
  \sum_{n=0}^{\infty}u_n(\x)\Psi_n'(t)
  +
  \sum_{n=0}^{\infty}
  \bigl(H(\x)\cdot\nabla u_n(\x)\bigr)\Psi_n(t)
  +
  \frac{
    \displaystyle
    f\left(\x,\sum_{n=0}^{\infty}u_n(\x)\Psi_n(t)\right)
  }{
    \displaystyle
    c\left(\x,\sum_{n=0}^{\infty}u_n(\x)\Psi_n(t)\right)
  }
  \\
  +
  \frac{1}{
    \displaystyle
    c\left(\x,\sum_{n=0}^{\infty}u_n(\x)\Psi_n(t)\right)}
  \sum_{n=0}^{\infty}u_n(\x)
  \int_0^t\alpha(\x,t-t')\Psi_n(t')\,dt'
  =0,
  \label{eq:infinite-expanded-pde}
\end{multline}
for all $(\x, t) \in Q_T.$
For each $m\geq0$, we multiply both sides of
\eqref{eq:infinite-expanded-pde} by $e^{-2t}\Psi_m(t)$ and integrate
with respect to $t$ over $(0,T)$. Using the orthonormality of
$\{\Psi_n\}_{n\geq0}$, we obtain
\begin{multline}
  H(\x)\cdot\nabla u_m(\x)
  +\sum_{n=0}^{\infty} 
  S_{mn}\bigl(\x,\{u_j\}_{j\geq0}\bigr)u_n(\x)
  +
  \\
  \sum_{n=0}^{\infty}
  M_{mn}\bigl(\x,\{u_j\}_{j\geq0}\bigr)u_n(\x)
  +F_m\bigl(\x,\{u_j\}_{j\geq0}\bigr)
  =0.
  \label{eq:projected-infinite-pde}
\end{multline}
for all $\x\in\Omega$ and $m\geq0$, where, for $m,n\geq0$,
\begin{equation}
  S_{mn}\bigl(\x,\{u_j\}_{j\geq0}\bigr)
  :=
  \int_0^T
  \frac{e^{-2t}\Psi_n'(t)\Psi_m(t)}
  {\displaystyle
   c\left(\x,\sum_{j=0}^{\infty}u_j(\x)\Psi_j(t)\right)}
  \,dt,
  \label{eq:infinite-Smn}
\end{equation}
\begin{equation}
  M_{mn}\bigl(\x,\{u_j\}_{j\geq0}\bigr)
  :=
  \int_0^T
  \frac{e^{-2t}\Psi_m(t)}
  {\displaystyle
   c\left(\x,\sum_{j=0}^{\infty}u_j(\x)\Psi_j(t)\right)}
  \left(
    \int_0^t\alpha(\x,t-t')\Psi_n(t')\,dt'
  \right)dt,
  \label{eq:infinite-Mmn}
\end{equation}
and
\begin{equation}
  F_m\bigl(\x,\{u_j\}_{j\geq0}\bigr)
  :=
  \int_0^T
  \frac{\displaystyle
    e^{-2t}
    f\left(\x,\sum_{j=0}^{\infty}u_j(\x)\Psi_j(t)\right)}
  {\displaystyle
    c\left(\x,\sum_{j=0}^{\infty}u_j(\x)\Psi_j(t)\right)}
  \Psi_m(t)\,dt.
  \label{eq:infinite-Fm}
\end{equation}

In computation, we approximate
\eqref{eq:projected-infinite-pde} by retaining the first $N+1$
components $u_0,\ldots,u_N$ and neglecting the tail
$\{u_j\}_{j>N}$. Since the series in
\eqref{eq:infinite-time-expansions} converges, the neglected tail tends
to zero in the corresponding weighted norm as $N\to\infty$. For the
resulting truncated series, differentiation and summation commute
algebraically.

The truncated version of \eqref{eq:projected-infinite-pde} is
\begin{multline}
  H(\x)\cdot\nabla u_m(\x)
  +\sum_{n=0}^{N}
  S_{mn}\bigl(\x,\{u_j\}_{j=0}^{N}\bigr)u_n(\x)
  \\
  +\sum_{n=0}^{N}
  M_{mn}\bigl(\x,\{u_j\}_{j=0}^{N}\bigr)u_n(\x)
  +F_m\bigl(\x,\{u_j\}_{j=0}^{N}\bigr)
  =0.
  \label{eq:truncated-projected-pde}
\end{multline}
for all $\x\in\Omega$ and $m=0,\ldots,N$.

The inflow data and outflow observation are projected in the same basis.
Projecting $u=g$ on $\Gamma_-\times(0,T)$ gives
\begin{equation}
  \label{eq:g_m}
  g_m(\x)
  =
  \int_0^T e^{-2t}g(\x,t)\Psi_m(t)\,dt,
  \qquad \x\in\Gamma_-.
\end{equation}
Similarly, under the full outflow observation assumption
$\Gamma_{\mathrm{obs}}=\Gamma_+$, the measured data $u=h$ on
$\Gamma_+\times(0,T)$ gives
\begin{equation}
  \label{eq:h_m}
  h_m(\x)
  =
  \int_0^T e^{-2t}h(\x,t)\Psi_m(t)\,dt,
  \qquad \x\in\Gamma_+.
\end{equation}

We introduce the modal vector
\begin{equation}
  U(\x)
  =
  \bigl(u_0(\x),\ldots,u_N(\x)\bigr)^\top
  \in\mathbb{R}^{N+1}.
  \label{eq:U-vector}
\end{equation}
Therefore, the boundary conditions for the modal vector are
\begin{equation}
  \begin{cases}
  U(\x)
  =
  \bigl(g_0(\x),\ldots,g_N(\x)\bigr)^\top,
  & \x\in\Gamma_-,
  \\[4pt]
  U(\x)
  =
  \bigl(h_0(\x),\ldots,h_N(\x)\bigr)^\top,
  & \x\in\Gamma_+.
  \end{cases}
  \label{eq:modal-vector-boundary}
\end{equation}

For
\[
  \Phi=(\phi_0,\ldots,\phi_N)^\top
  \qquad\text{and}\qquad
  V=(v_0,\ldots,v_N)^\top,
\]
we collect the frozen nonlinear terms into the vector
$\mathcal N(\x;\Phi,V)$, whose $m^{\rm th}$ component is
\begin{multline}
  \bigl[\mathcal N(\x;\Phi,V)\bigr]_m
  :=
  \sum_{n=0}^{N}
  S_{mn}\bigl(\x,\{\phi_j\}_{j=0}^{N}\bigr)v_n(\x)
  \\
  +
  \sum_{n=0}^{N}
  M_{mn}\bigl(\x,\{\phi_j\}_{j=0}^{N}\bigr)v_n(\x)
  +
  F_m\bigl(\x,\{\phi_j\}_{j=0}^{N}\bigr),
  \qquad m=0,\ldots,N.
  \label{eq:frozen-nonlinearity}
\end{multline}
In particular, the nonlinear terms in the reduced equation are given
by $\mathcal N(\x;U,U)$.

Combining \eqref{eq:truncated-projected-pde}, \eqref{eq:g_m}, and
\eqref{eq:h_m},
we obtain, for each $m = 0,\dots,N$, the following coupled nonlinear
first-order spatial system with boundary conditions:
\begin{equation}
  \label{eq:reduced-full-system}
  \begin{cases}
    \displaystyle
    H(\x)\cdot\nabla u_m(\x)
    +\bigl[\mathcal N(\x;U,U)\bigr]_m
    = 0,
    & \x\in\Omega,
    \\[8pt]
    \displaystyle
    u_m(\x)
    = g_m(\x),
    & \x\in\Gamma_-,
    \\[8pt]
    \displaystyle
    u_m(\x)
    = h_m(\x),
    & \x\in\Gamma_+.
  \end{cases}
\end{equation}

The standard Picard iteration for solving
\eqref{eq:reduced-full-system} starts from an initial guess $U^{(0)}$.
The nonlinear terms are frozen at $U^{(0)}$, and the resulting linear
system is solved to obtain $U^{(1)}$. More precisely, writing
$U^{(k)}=(u_0^{(k)},\ldots,u_N^{(k)})^\top$, we compute $U^{(1)}$ from
\begin{equation}
  \begin{cases}
  \displaystyle
  H(\x)\cdot\nabla u_m^{(1)}(\x)
  +\bigl[\mathcal N(\x;U^{(0)},U^{(1)})\bigr]_m
  =0,
  & \x\in\Omega,
  \\[6pt]
  u_m^{(1)}(\x)=g_m(\x),
  & \x\in\Gamma_-,
  \\[4pt]
  u_m^{(1)}(\x)=h_m(\x),
  & \x\in\Gamma_+,
  \end{cases}
  \label{eq:first-standard-picard-step}
\end{equation}
for $m=0,\ldots,N$.
Thus, the coefficients are evaluated using $U^{(0)}$, whereas
$U^{(1)}$ is the unknown of this linear system. Repeating this
procedure produces a sequence $\{U^{(k)}\}_{k\geq0}$. Its convergence
depends both on the initialization $U^{(0)}$ and on how the resulting
linear system at each iteration is solved. A natural approach is to
solve the linear system \eqref{eq:first-standard-picard-step}, and the
corresponding linear systems at subsequent iterations, by an
unweighted least-squares method. Such a procedure, however, does not in
general guarantee that the resulting Picard sequence converges to a solution of
\eqref{eq:reduced-full-system}. We therefore introduce a suitable
Carleman weight into the least-squares functional. The Carleman
estimate established below implies that the resulting
Carleman--Picard map is contractive on the admissible set. Consequently,
the convergence does not require $U^{(0)}$ to be close to the true
solution, provided that $U^{(0)}$ belongs to the admissible set and the
assumptions of the convergence theorem are satisfied.

\section{The Carleman contraction principle}
\label{sec:contraction}

We now define a fixed-point reconstruction scheme and use the
vector-valued Carleman estimate from
Corollary~\ref{cor:system-carleman} to establish its contractivity.

Let $s>d/2+1$. Then
\[
  H^s(\Omega)\hookrightarrow W^{1,\infty}(\Omega),
\]
and there exists a constant $C_{\mathrm{emb}}>0$ such that
\[
  \|V\|_{[W^{1,\infty}(\Omega)]^{N+1}}
  \leq
  C_{\mathrm{emb}}\|V\|_{[H^s(\Omega)]^{N+1}}.
\]
Recalling $g_j$ and $h_j$ from \eqref{eq:g_m} and \eqref{eq:h_m},
respectively, we write
\[
  G(\x):=(g_0(\x),\ldots,g_N(\x))^\top,
  \qquad
  D^{\mathrm{obs}}(\x):=(h_0(\x),\ldots,h_N(\x))^\top
\]
for the projected inflow data and outflow observation, respectively.

For a fixed constant $M>0$, let $B_M$ denote the closed ball of radius
$M$ in $[H^s(\Omega)]^{N+1}$:
\begin{equation}
\label{eq:closed-ball}
B_M
:=
\left\{
V\in[H^s(\Omega)]^{N+1}:
\|V\|_{[H^s(\Omega)]^{N+1}}\leq M
\right\}.
\end{equation}
The set $B_M$ is nonempty, convex, bounded, and weakly closed.

For each frozen vector
$\Phi=(\phi_0,\ldots,\phi_N)^\top\in B_M$, and for parameters
$\lambda>0$ and $\varepsilon>0$, define
\begin{multline}
  J_{\Phi}^{\lambda,\varepsilon}(U)
  :=
  \int_{\Omega}
  e^{2\lambda\varphi_*}
  \left|
    H(\x)\cdot\nabla U(\x)
    +\mathcal N(\x;\Phi,U)
  \right|^2\,d\x
  \\
  +
  \lambda^2
  \int_{\Gamma_-}
  e^{2\lambda\varphi_*}
  |U-G|^2\,dS
  +
  \lambda^2
  \int_{\Gamma_+}
  e^{2\lambda\varphi_*}
  |U-D^{\mathrm{obs}}|^2\,dS
  +
  \varepsilon
  \|U\|^2_{[H^s(\Omega)]^{N+1}},
  \label{eq:backward-functional}
\end{multline}
for $U\in B_M$.
Here, the first term penalizes the residual of the frozen equation, the
next two terms fit the projected boundary data on $\Gamma_-$ and
$\Gamma_+$, respectively, and the last term is a Tikhonov
regularization term.

We next verify that the frozen Carleman-weighted minimization problem
is well posed.
\begin{proposition}
\label{prop:backward-minimizer}
For each $\Phi\in B_M$, $\lambda>0$, and $\varepsilon>0$, the
functional $J_{\Phi}^{\lambda,\varepsilon}$ admits a unique minimizer in
$B_M$.
\end{proposition}

\begin{proof}
This follows from the standard direct method. Indeed, $B_M$ is a
nonempty, convex, and weakly compact subset of
$[H^s(\Omega)]^{N+1}$. For fixed $\Phi$, the functional
$J_{\Phi}^{\lambda,\varepsilon}$ is weakly lower semicontinuous because
the frozen residual is affine and continuous in $U$, the trace map is
continuous, and the Hilbert norm is weakly lower semicontinuous.
Therefore, $J_{\Phi}^{\lambda,\varepsilon}$ attains its minimum on
$B_M$. Since $\varepsilon>0$, the Tikhonov term is strictly convex;
hence the minimizer is unique.
\end{proof}

By Proposition~\ref{prop:backward-minimizer}, the Picard map
\[
  \mathcal T^{\lambda,\varepsilon}:B_M\to B_M
\]
is well-defined by
\begin{equation}
  \label{eq:backward-picard-map}
  \mathcal T^{\lambda,\varepsilon}(\Phi)
  :=
  \operatorname*{argmin}_{U\in B_M}
  J_{\Phi}^{\lambda,\varepsilon}(U).
\end{equation}

\begin{lemma}[Lipschitz continuity of the nonlinearity]
\label{lem:nonlinearity-lipschitz}
Assume that
$c,f\in C^1(\overline\Omega\times\mathbb R)$,
$c$ satisfies \eqref{eq:nondegenerate-velocity}, and
$\alpha\in L^\infty(\Omega\times(0,T))$.
Then there exists $C_{\mathcal N}>0$, depending only on
$\Omega$, $s$, $M$, $N$, $T$, $\kappa_0$,
$\{\Psi_n\}_{n=0}^{N}$, $\|\alpha\|_{L^\infty(\Omega\times(0,T))}$,
and the relevant $C^1$ bounds of $c$ and $f$, such that, for all
$\Phi,\Psi,U,V\in B_M$ and almost every $\x\in\Omega$,
\begin{equation}
  |\mathcal N(\x;\Phi,U)-\mathcal N(\x;\Psi,V)|
  \leq
  C_{\mathcal N}
  \big(
  |\Phi(\x)-\Psi(\x)|
  +
  |U(\x)-V(\x)|
  \big).
  \label{eq:nonlinearity-lipschitz}
\end{equation}
\end{lemma}

\begin{proof}
Since $s>d/2+1$, the continuous Sobolev embedding gives
\[
  \|V\|_{[C^0(\overline\Omega)]^{N+1}}
  \leq C_{\mathrm{emb}}M,
  \qquad V\in B_M.
\]
Thus the coefficient vectors at which $\mathcal N$ is evaluated remain
in a fixed bounded subset of $\mathbb R^{N+1}$. Since $\mathcal N$
contains no spatial derivatives and is continuously differentiable
with respect to the coefficient vectors $\Phi$ and $U$, its derivative
is uniformly bounded on this set. The mean value theorem yields
\eqref{eq:nonlinearity-lipschitz}.
\end{proof}

The preceding estimates allow us to prove that the Carleman--Picard
map is contractive on the admissible set for sufficiently large
Carleman parameter.
We equip $B_M$ with the weighted norm
\begin{equation}
  \label{eq:backward-carleman-norm}
  \|Z\|^2_{\lambda,\varepsilon,\varphi_*}
  :=
  \int_\Omega e^{2\lambda\varphi_*}|Z|^2\,d\x
  +
  \int_{\Gamma_+}
  e^{2\lambda\varphi_*}|Z|^2\,dS
  +
  \frac{\varepsilon}{\lambda}
  \|Z\|^2_{[H^s(\Omega)]^{N+1}}.
\end{equation}
For fixed $\lambda>0$ and $\varepsilon>0$, this norm is equivalent to the
standard $[H^s(\Omega)]^{N+1}$ norm. Therefore it induces an equivalent
metric on $B_M$.

\begin{theorem}
  \label{thm:backward-contraction}
  Suppose that the hypotheses of Corollary~\ref{cor:system-carleman}
  hold, and let $B_M$ be the closed ball defined in
  \eqref{eq:closed-ball}. Assume also that the hypotheses of
  Lemma~\ref{lem:nonlinearity-lipschitz} hold.
  Then there exist constants $\lambda_1>0$ and $C>0$, depending only on
  $\Omega$, $H$, $\varphi_*$, $\mu_0$, $N$, $M$, $s$, and
  $C_{\mathcal N}$, such that, for every fixed $\varepsilon>0$, all
  $\lambda\geq\lambda_1$, and all $\Phi,\Psi\in B_M$,
  \begin{equation}
    \left\|
      \mathcal T^{\lambda,\varepsilon}(\Phi)
      -
      \mathcal T^{\lambda,\varepsilon}(\Psi)
    \right\|_{\lambda,\varepsilon,\varphi_*}
    \leq
    \sqrt{\frac{C}{\lambda}}\,
    \|\Phi-\Psi\|_{\lambda,\varepsilon,\varphi_*}.
    \label{eq:backward-contraction}
  \end{equation}
  In particular, if
  \[
      \lambda>\max\{\lambda_1,C\},
  \]
  then $\mathcal T^{\lambda,\varepsilon}$ is a strict contraction on
  $B_M$.
\end{theorem}

\begin{proof}
Throughout the proof, $C>0$ denotes a generic constant depending only
on the parameters listed in the statement. Set
$U=\mathcal T^{\lambda,\varepsilon}(\Phi)$,
$V=\mathcal T^{\lambda,\varepsilon}(\Psi)$, and $W=U-V$.

Since $U$ minimizes $J_\Phi^{\lambda,\varepsilon}$ over $B_M$ and
$U-\theta W=(1-\theta)U+\theta V\in B_M$ for $0<\theta<1$, we obtain
\begin{multline}
  \int_\Omega e^{2\lambda\varphi_*}
  [H\cdot\nabla U+\mathcal N(\x;\Phi,U)]\cdot
  [H\cdot\nabla W+\mathcal N(\x;\Phi,U)-\mathcal N(\x;\Phi,V)]\,d\x
  \\
  +\lambda^2\int_{\Gamma_-}e^{2\lambda\varphi_*}(U-G)\cdot W\,dS
  +\lambda^2\int_{\Gamma_+}e^{2\lambda\varphi_*}
  (U-D^{\mathrm{obs}})\cdot W\,dS
  +\varepsilon\langle U,W\rangle_{[H^s]^{N+1}}\leq0.
  \label{eq:opt-U}
\end{multline}
Similarly, using $V+\theta W=(1-\theta)V+\theta U\in B_M$, we obtain
\begin{multline}
  \int_\Omega e^{2\lambda\varphi_*}
  [H\cdot\nabla V+\mathcal N(\x;\Psi,V)]\cdot
  [H\cdot\nabla W+\mathcal N(\x;\Psi,U)-\mathcal N(\x;\Psi,V)]\,d\x
  \\
  +\lambda^2\int_{\Gamma_-}e^{2\lambda\varphi_*}(V-G)\cdot W\,dS
  +\lambda^2\int_{\Gamma_+}e^{2\lambda\varphi_*}
  (V-D^{\mathrm{obs}})\cdot W\,dS
  +\varepsilon\langle V,W\rangle_{[H^s]^{N+1}}\geq0.
  \label{eq:opt-V}
\end{multline}
Subtracting \eqref{eq:opt-U} from \eqref{eq:opt-V} gives
\begin{multline}
  \int_\Omega e^{2\lambda\varphi_*}
  [H\cdot\nabla V+\mathcal N(\x;\Psi,V)]\cdot
  [H\cdot\nabla W+\mathcal N(\x;\Psi,U)-\mathcal N(\x;\Psi,V)]\,d\x
  \\
  -\int_\Omega e^{2\lambda\varphi_*}
  [H\cdot\nabla U+\mathcal N(\x;\Phi,U)]\cdot
  [H\cdot\nabla W+\mathcal N(\x;\Phi,U)-\mathcal N(\x;\Phi,V)]\,d\x
  \\
  \geq\lambda^2\int_{\partial\Omega}e^{2\lambda\varphi_*}|W|^2\,dS
  +\varepsilon\|W\|^2_{[H^s]^{N+1}}.
  \label{eq:pre-carleman}
\end{multline}

Set $A=H\cdot\nabla W+\mathcal N(\x;\Phi,U)-\mathcal N(\x;\Phi,V)$,
$P=\mathcal N(\x;\Psi,V)-\mathcal N(\x;\Phi,V)$, and
$Q=\mathcal N(\x;\Psi,U)-\mathcal N(\x;\Psi,V)
-\mathcal N(\x;\Phi,U)+\mathcal N(\x;\Phi,V)$.
The difference of the interior integrands in \eqref{eq:pre-carleman} is
$-|A|^2+[H\cdot\nabla U+\mathcal N(\x;\Phi,U)]\cdot Q
-A\cdot Q+P\cdot A+P\cdot Q$.
Lemma~\ref{lem:nonlinearity-lipschitz}, the affine dependence of
$\mathcal N$ on its last argument, and the boundedness of $B_M$ give
$|P|\leq C|\Phi-\Psi|$, $|Q|\leq C|\Phi-\Psi||W|$, and
$|H\cdot\nabla U+\mathcal N(\x;\Phi,U)|\leq C$.
Using $|a\cdot b|\leq\delta_1|a|^2+(4\delta_1)^{-1}|b|^2$ and choosing
$\delta_1>0$ sufficiently small, we obtain
\begin{multline}
  \int_\Omega e^{2\lambda\varphi_*}
  |H\cdot\nabla W+\mathcal N(\x;\Phi,U)-\mathcal N(\x;\Phi,V)|^2\,d\x
  +\lambda^2\int_{\partial\Omega}e^{2\lambda\varphi_*}|W|^2\,dS
  \\
  +\varepsilon\|W\|^2_{[H^s]^{N+1}}
  \leq C\int_\Omega e^{2\lambda\varphi_*}
  (|\Phi-\Psi|^2+|W|^2)\,d\x.
  \label{eq:pre-carleman-estimate}
\end{multline}
Using $|a+b|^2\geq\frac12|a|^2-|b|^2$, we deduce
\begin{multline}
  \int_\Omega e^{2\lambda\varphi_*}|H\cdot\nabla W|^2\,d\x
  +\lambda^2\int_{\partial\Omega}e^{2\lambda\varphi_*}|W|^2\,dS
  +\varepsilon\|W\|^2_{[H^s]^{N+1}}
  \\
  \leq C\int_\Omega e^{2\lambda\varphi_*}(|\Phi-\Psi|^2+|W|^2)\,d\x
  +C\int_\Omega e^{2\lambda\varphi_*}
  |\mathcal N(\x;\Phi,U)-\mathcal N(\x;\Phi,V)|^2\,d\x.
  \label{eq:pre-carleman-principal}
\end{multline}
Lemma~\ref{lem:nonlinearity-lipschitz} therefore gives
\begin{multline}
  \int_\Omega e^{2\lambda\varphi_*}|H\cdot\nabla W|^2\,d\x
  +\lambda^2\int_{\partial\Omega}e^{2\lambda\varphi_*}|W|^2\,dS
  +\varepsilon\|W\|^2_{[H^s]^{N+1}}
  \leq C\int_\Omega e^{2\lambda\varphi_*}
  (|\Phi-\Psi|^2+|W|^2)\,d\x.
  \label{eq:pre-carleman-lipschitz}
\end{multline}

For $\lambda\geq\lambda_0$, Corollary~\ref{cor:system-carleman} gives
\begin{equation}
  \lambda\int_\Omega e^{2\lambda\varphi_*}|W|^2\,d\x
  \leq C\int_\Omega e^{2\lambda\varphi_*}|H\cdot\nabla W|^2\,d\x
  +C\lambda\int_{\partial\Omega}e^{2\lambda\varphi_*}|W|^2\,dS.
  \label{eq:carleman-W}
\end{equation}
Taking $\lambda_1\geq\max\{\lambda_0,1\}$ and combining the last two
estimates, we obtain
\begin{multline}
  \lambda\int_\Omega e^{2\lambda\varphi_*}|W|^2\,d\x
  +\int_\Omega e^{2\lambda\varphi_*}|H\cdot\nabla W|^2\,d\x
  +\lambda^2\int_{\partial\Omega}e^{2\lambda\varphi_*}|W|^2\,dS
  \\
  +\varepsilon\|W\|^2_{[H^s]^{N+1}}
  \leq C\int_\Omega e^{2\lambda\varphi_*}|\Phi-\Psi|^2\,d\x
  +C\int_\Omega e^{2\lambda\varphi_*}|W|^2\,d\x.
  \label{eq:absorb-pre}
\end{multline}
For sufficiently large $\lambda_1$, the last term is absorbed, and the
nonnegative derivative term may be omitted. Hence
\begin{equation}
  \lambda\int_\Omega e^{2\lambda\varphi_*}|W|^2\,d\x
  +\lambda^2\int_{\partial\Omega}e^{2\lambda\varphi_*}|W|^2\,dS
  +\varepsilon\|W\|^2_{[H^s]^{N+1}}
  \leq C\int_\Omega e^{2\lambda\varphi_*}|\Phi-\Psi|^2\,d\x.
  \label{eq:W-L2}
\end{equation}
Dividing by $\lambda$, using $\lambda\geq1$, and adding the nonnegative
terms $\frac{C}{\lambda}\int_{\Gamma_+}e^{2\lambda\varphi_*}
|\Phi-\Psi|^2\,dS$ and
$\frac{C\varepsilon}{\lambda^2}\|\Phi-\Psi\|^2_{[H^s]^{N+1}}$ to the
right-hand side, we get
\begin{multline}
  \int_\Omega e^{2\lambda\varphi_*}|W|^2\,d\x
  +\int_{\Gamma_+}e^{2\lambda\varphi_*}|W|^2\,dS
  +\frac{\varepsilon}{\lambda}\|W\|^2_{[H^s]^{N+1}}
  \\
  \leq\frac{C}{\lambda}\big(
  \int_\Omega e^{2\lambda\varphi_*}|\Phi-\Psi|^2\,d\x
  +\int_{\Gamma_+}e^{2\lambda\varphi_*}|\Phi-\Psi|^2\,dS
  +\frac{\varepsilon}{\lambda}\|\Phi-\Psi\|^2_{[H^s]^{N+1}}
  \big).
  \label{eq:contraction-pre}
\end{multline}
Thus
$\|U-V\|_{\lambda,\varepsilon,\varphi_*}
\leq\sqrt{C/\lambda}\,
\|\Phi-\Psi\|_{\lambda,\varepsilon,\varphi_*}$, proving
\eqref{eq:backward-contraction}.
\end{proof}

As a direct consequence, the Picard sequence converges globally within
$B_M$.
\begin{corollary}
  \label{cor:backward-picard}
  Let $\varepsilon>0$ and choose $\lambda\geq\lambda_1$ so that
  $\mu:=\sqrt{C/\lambda}<1$. For arbitrary $U^{(0)}\in B_M$, define
  $U^{(k+1)}=\mathcal T^{\lambda,\varepsilon}(U^{(k)})$.
  Then $U^{(k)}$ converges in
  $\|\cdot\|_{\lambda,\varepsilon,\varphi_*}$ to the unique fixed point
  $U^{\lambda,\varepsilon}\in B_M$, and
  \[
    \|U^{(k)}-U^{\lambda,\varepsilon}\|_{\lambda,\varepsilon,\varphi_*}
    \leq\mu^k
    \|U^{(0)}-U^{\lambda,\varepsilon}\|_{\lambda,\varepsilon,\varphi_*}.
  \]
\end{corollary}
\begin{proof}
This follows from Theorem~\ref{thm:backward-contraction} and the Banach
fixed-point theorem because $B_M$ is complete in the equivalent weighted
metric.
\end{proof}

\begin{Remark}
Once a fixed point
\[
  U=(u_0,\ldots,u_N)^\top
\]
of the Carleman--Picard map has been computed, the truncated initial
condition is reconstructed by
\begin{equation}
  u^{0,N}(\x)
  =
  \sum_{n=0}^{N}u_n(\x)\Psi_n(0).
  \label{eq:initial-reconstruction}
\end{equation}
\end{Remark}

\section{Stability with respect to noisy data}
\label{sec:stability}
Assume the hypotheses of Theorem~\ref{thm:backward-contraction}. Let
$U^*\in B_M$ satisfy
$H\cdot\nabla U^*+\mathcal N(\x;U^*,U^*)=0$ in $\Omega$ and
$U^*|_{\Gamma_-}=G$, and define the exact projected outflow data by
$D^*:=U^*|_{\Gamma_+}$. We model the measurement error as
multiplicative noise and assume
\begin{equation}
  \|D^{\mathrm{obs}}-D^*\|_{[L^2(\Gamma_+)]^{N+1}}
  \leq\delta\|D^*\|_{[L^2(\Gamma_+)]^{N+1}},
  \qquad 0<\delta<1.
  \label{eq:multiplicative-noise}
\end{equation}

\begin{theorem}[Stability with respect to noisy data]
\label{thm:backward-stability}
Let $\mathcal T^{\lambda,\varepsilon}:B_M\to B_M$ be the
Carleman--Picard map defined in \eqref{eq:backward-picard-map}, with
$D^{\mathrm{obs}}$ satisfying \eqref{eq:multiplicative-noise}. Then
$\mathcal T^{\lambda,\varepsilon}$ is a contraction on $B_M$ for all
sufficiently large $\lambda$. Its fixed point
$U^{\lambda,\varepsilon,\delta}\in B_M$ satisfies
\begin{multline}
  \int_\Omega e^{2\lambda\varphi_*}
  |U^{\lambda,\varepsilon,\delta}-U^*|^2\,d\x
  +\lambda\int_{\partial\Omega}e^{2\lambda\varphi_*}
  |U^{\lambda,\varepsilon,\delta}-U^*|^2\,dS
  \\
  +\frac{\varepsilon}{\lambda}
  \|U^{\lambda,\varepsilon,\delta}-U^*\|^2_{[H^s(\Omega)]^{N+1}}
  \leq
  \frac{C\varepsilon}{\lambda}\|U^*\|^2_{[H^s(\Omega)]^{N+1}}
  +C\lambda\int_{\Gamma_+}e^{2\lambda\varphi_*}
  |D^{\mathrm{obs}}-D^*|^2\,dS.
  \label{eq:backward-stability}
\end{multline}
\end{theorem}

\begin{proof}
Set $U^\delta=U^{\lambda,\varepsilon,\delta}$ and $W=U^\delta-U^*$.
Since $U^\delta$ is a fixed point of
$\mathcal T^{\lambda,\varepsilon}$, it minimizes
$J_{U^\delta}^{\lambda,\varepsilon}$ over $B_M$. Moreover,
$U^\delta-\theta W=(1-\theta)U^\delta+\theta U^*\in B_M$ for
$0<\theta<1$. Proceeding as in the derivation of \eqref{eq:opt-U}, with
$\Phi=U^\delta$, we obtain
\begin{multline}
  \int_\Omega e^{2\lambda\varphi_*}
  [H\cdot\nabla U^\delta+\mathcal N(\x;U^\delta,U^\delta)]\cdot
  [H\cdot\nabla W+\mathcal N(\x;U^\delta,U^\delta)
  -\mathcal N(\x;U^\delta,U^*)]\,d\x
  \\
  +\lambda^2\int_{\Gamma_-}e^{2\lambda\varphi_*}
  (U^\delta-G)\cdot W\,dS
  +\lambda^2\int_{\Gamma_+}e^{2\lambda\varphi_*}
  (U^\delta-D^{\mathrm{obs}})\cdot W\,dS
  \\
  +\varepsilon\langle U^\delta,W\rangle_{[H^s]^{N+1}}
  \leq0.
  \label{eq:stability-variational}
\end{multline}
On the other hand, the equation and exact boundary conditions for
$U^*$ imply
\begin{multline}
  \int_\Omega e^{2\lambda\varphi_*}
  [H\cdot\nabla U^*+\mathcal N(\x;U^*,U^*)]\cdot
  [H\cdot\nabla W+\mathcal N(\x;U^\delta,U^\delta)
  -\mathcal N(\x;U^\delta,U^*)],d\x
  \\
  +\lambda^2\int_{\Gamma_-}e^{2\lambda\varphi_*}
  (U^*-G)\cdot W\,dS
  +\lambda^2\int_{\Gamma_+}e^{2\lambda\varphi_*}
  (U^*-D^*)\cdot W\,dS
  =0.
  \label{eq:exact-state-identity}
\end{multline}
Subtracting \eqref{eq:exact-state-identity} from
\eqref{eq:stability-variational}, using
$U^\delta-U^*=W$, and applying
Lemma~\ref{lem:nonlinearity-lipschitz} and Young's inequality, we obtain
\begin{multline}
  \int_\Omega e^{2\lambda\varphi_*}|H\cdot\nabla W|^2\,d\x
  +\lambda^2\int_{\partial\Omega}e^{2\lambda\varphi_*}|W|^2\,dS
  +\varepsilon\|W\|^2_{[H^s]^{N+1}}
  \\
  \leq C\int_\Omega e^{2\lambda\varphi_*}|W|^2\,d\x
  +C\lambda^2\int_{\Gamma_+}e^{2\lambda\varphi_*}
  |D^{\mathrm{obs}}-D^*|^2\,dS
  +C\varepsilon\|U^*\|^2_{[H^s]^{N+1}}.
  \label{eq:stability-pre-carleman}
\end{multline}
The Carleman estimate and absorption for sufficiently large $\lambda$
yield
\begin{multline*}
  \lambda\int_\Omega e^{2\lambda\varphi_*}|W|^2\,d\x
  +\lambda^2\int_{\partial\Omega}e^{2\lambda\varphi_*}|W|^2\,dS
  +\varepsilon\|W\|^2_{[H^s]^{N+1}}
  \\
  \leq C\lambda^2\int_{\Gamma_+}e^{2\lambda\varphi_*}
  |D^{\mathrm{obs}}-D^*|^2\,dS
  +C\varepsilon\|U^*\|^2_{[H^s]^{N+1}}.
\end{multline*}
Dividing by $\lambda$, we obtain
\begin{multline*}
  \int_\Omega e^{2\lambda\varphi_*}|W|^2\,d\x
  +\lambda\int_{\partial\Omega}e^{2\lambda\varphi_*}|W|^2\,dS
  +\frac{\varepsilon}{\lambda}\|W\|^2_{[H^s]^{N+1}}
  \\
  \leq
  C\lambda\int_{\Gamma_+}e^{2\lambda\varphi_*}
  |D^{\mathrm{obs}}-D^*|^2\,dS
  +\frac{C\varepsilon}{\lambda}
  \|U^*\|^2_{[H^s]^{N+1}}.
\end{multline*}
Using \eqref{eq:multiplicative-noise} and
$e^{2\lambda\varphi_*}\leq
e^{2\lambda\|\varphi_*\|_{L^\infty(\Omega)}}$ proves
\eqref{eq:backward-stability}.
\end{proof}

\begin{Remark}
\label{rem:two-errors-backward}
Although the original inverse problem is ill-posed, the Lipschitz
stability estimate \eqref{eq:backward-stability} does not contradict
this fact because it is established for a truncated and regularized
problem. First, fixing the truncation index $N$ restricts the temporal
dependence of the reconstructed solution to
$\operatorname{span}\{\Psi_0,\ldots,\Psi_N\}$ and excludes the
higher-index temporal modes, including highly oscillatory behaviors,
from the reduced model. Second, the Tikhonov term with parameter
$\varepsilon$ provides additional regularization for the resulting
finite coupled system.

Consequently, for a fixed $N$ and a fixed sufficiently large
$\lambda$, estimate \eqref{eq:backward-stability} shows that the
reconstruction error depends Lipschitz continuously on the relative
noise level $\delta$, up to the Tikhonov regularization error of order
$\sqrt{\varepsilon/\lambda}$. This stability result therefore applies
only to the truncated and regularized reduced problem. In particular,
the stability constant may depend on $N$ and on the retained basis
functions, and no estimate uniform with respect to $N$ is established
here. Hence, the estimate cannot be directly extended to the original
untruncated problem by letting $N\to\infty$ and
$\varepsilon\to0$.
\end{Remark}

Theorems~\ref{thm:backward-contraction} and
\ref{thm:backward-stability} indicate that the fixed point of
$\mathcal T^{\lambda,\varepsilon}$ provides a stable approximation to
the solution of the reduced system
\eqref{eq:reduced-full-system}. Once this fixed point has been
computed, the truncated approximation $u^{0,N}$ of the unknown initial
state is recovered using \eqref{eq:initial-reconstruction}.
The complete reconstruction procedure is summarized in
Algorithm~\ref{alg:carleman-picard}.

\begin{algorithm}[H]
\caption{Carleman--Picard reconstruction of $u^0$}
\label{alg:carleman-picard}
\begin{algorithmic}[1]
\State  Choose the truncation index $N$, the Carleman parameter
$\lambda$, the regularization parameter $\varepsilon$, the stopping
tolerance $\tau$, and the Carleman weight function $\varphi_*$.
\State Construct $\{\Psi_n\}_{n=0}^N$ and project $g$ and
$h^{\rm obs}$ to obtain $G$ and $D^{\rm obs}$.\State Set the initializer $U^{(0)}=0$.
\State Set $k^*=K_{\max}$.
\For{$k=1,\ldots,K_{\max}$}
  \State Freeze the first argument of
$\mathcal N(\x;U^{(k-1)},V)$ at $U^{(k-1)}$.
  \State Compute $U^{(k)}$ by minimizing
  $J_{U^{(k-1)}}^{\lambda,\varepsilon}$ from
  \eqref{eq:backward-functional}.
  \If{$\frac{\|U^{(k)}-U^{(k-1)}\|_{\ell^2}}{
  \max\{\|U^{(k)}\|_{\ell^2},10^{-14}\}}<\tau$}
    \State Set $k^*=k$ and terminate.
  \EndIf
\EndFor
\State Recover
$u^{0,N}(\x)=\displaystyle\sum_{n=0}^N
u_n^{(k^*)}(\x)\Psi_n(0)$.
\end{algorithmic}
\end{algorithm}

\section{Numerical Study}
\label{sec:numerics}
This section presents a two-dimensional numerical realization of the
Carleman--Picard reconstruction method and evaluates its performance for
three representative initial profiles under several levels of outflow
noise.
\subsection{Computational Setup}

We consider the square domain $\Omega=(-1,1)^2$. The spatially varying
advection field and the nonlinear coefficients are
\begin{align}
  H(x,y)
  &=
  \left(
  1+0.25\sin(\pi y)\cos(\pi x),
  0.5+0.20\cos(\pi x)\cos(\pi y)
  \right)^\top,
  \notag\\
  c(x,y,s)
  &=1+0.15\sin(\pi x)\cos(\pi y)+0.2s^2,
  \notag
  \\
  f(x,y,s)
  &=0.02\bigl(1+0.25\cos(\pi x)\sin(\pi y)\bigr)s,
  \notag
  \\
  \alpha(x,y,s)
  &=0.02\bigl(1+0.20\sin(\pi x)\cos(\pi y)\bigr)e^{-s}.
  \label{eq:num-coeffs}
\end{align}
for all $(x,y)\in\Omega$ and $s\in\mathbb R$. In particular,
$c(x,y,s)\geq0.85$, so
\eqref{eq:nondegenerate-velocity} holds with $\kappa_0=0.85$.
Moreover, $H_x\geq0.75$ and $H_y\geq0.30$. Hence the field is
nontrapping, and every forward integral curve exits through the right or
top side of the square. The computations use $T=1.5$. 

Synthetic boundary data are generated by solving the forward problem on
a uniform $61\times61$ Cartesian grid, with
\[
  \Delta x=\Delta y=\frac{2}{60}.
\]
The transport term is discretized by a first-order upwind scheme, and
time integration is performed by explicit Euler with the CFL-stable
step
\[
  \Delta t_{\rm CFL}
  =
  \frac{0.4}{\displaystyle
  c_{\max}\max_{\overline\Omega}
  \left(\frac{|H_x|}{\Delta x}+\frac{|H_y|}{\Delta y}\right)},
  \qquad c_{\max}=1.35.
\]
We set $N_t=\lceil T/\Delta t_{\rm CFL}\rceil$ and then use the uniform
step $\Delta t=T/N_t$. Let $t_n=n\Delta t$ and write
$u^n(x,y)\approx u(x,y,t_n)$.

For the exponential memory kernel, let
\[
  A^n(x,y)
  \approx
  \int_0^{t_n}
  \alpha(x,y,t_n-t')u(x,y,t')\,dt',
  \qquad
  A^0(x,y)=0.
\]
Using a piecewise-constant approximation of $u$ on each time interval,
the memory term is updated recursively by
\[
  A^{n+1}(x,y)
  =
  e^{-\Delta t}A^n(x,y)
  +
  \alpha(x,y,0)\bigl(1-e^{-\Delta t}\bigr)u^n(x,y).
\]
The forward state is then advanced according to
\begin{equation}
  u^{n+1}
  =
  u^n
  +
  \Delta t
  \left[
    -c(x,y,u^n)H\cdot\nabla_h u^n
    -f(x,y,u^n)
    -A^n
  \right].
  \label{eq:forward-update}
\end{equation}
In all numerical tests, we take $g=0$ on $\Gamma_-$. The inflow
boundary values used by the upwind discretization are therefore fixed
at zero throughout the forward computation, and no additional
post-processing correction is required after each time step.

\paragraph{Noise model.}
Let $\xi$ be an array of independent standard Gaussian random variables.
The noisy outflow data are generated by
\[
  h^{\rm obs}
  =
  h
  \left(
  1
  +\delta
  \frac{\|h\|_{L^2(\Gamma_+\times(0,T))}}
       {\|h\xi\|_{L^2(\Gamma_+\times(0,T))}}
  \xi
  \right).
\]
Consequently,
\[
  \frac{\|h^{\rm obs}-h\|_{L^2(\Gamma_+\times(0,T))}}
       {\|h\|_{L^2(\Gamma_+\times(0,T))}}
  =\delta.
\]

\subsection{Implementation}
\label{sec:numerical-implementation}

\paragraph{Time-dimensional reduction and the choice of $N$.}
For each candidate $N\in\{1,\ldots,25\}$, we construct
$\{\Psi_n\}_{n=0}^{N}$ from the explicit formula
\eqref{eq:legendre-exponential-basis} and project the noisy observation
onto this basis by trapezoidal quadrature. For each $N$, define the
relative projection residual
\[
  r_N
  :=
  \frac{\|P_Nh^{\rm obs}-h^{\rm obs}\|_{L^2(\Gamma_+\times(0,T))}}
  {\|h^{\rm obs}\|_{L^2(\Gamma_+\times(0,T))}}.
\]
We plot $r_N$ against $N$ and choose the smallest $N$ at which the curve
begins to level off. Numerically, we require the decreases in $r_N$ over
the next three values of $N$ to be no greater than $1\%$ of the total
decrease observed over the tested range. Taking the absolute minimum of
$r_N$ would generally favor the largest tested $N$ and fit unnecessary
high-frequency components of the noisy data. The plateau criterion is
performed directly from $h^{\rm obs}$.
After $N$ is fixed, the projected inflow and outflow coefficients are
computed by the corresponding uniform-grid quadrature:
\[
G_m(\x)
=
\int_0^T e^{-2t}g(\x,t)\Psi_m(t)\,dt,
\qquad \x\in\Gamma_-,
\]
and
\[
D_m^{\rm obs}(\x)
  =
  \int_0^T e^{-2t}h^{\rm obs}(\x,t)\Psi_m(t)\,dt,
\qquad \x\in\Gamma_+.
\]

\paragraph{Carleman--Picard solver.}
Let $\mathbf U^{(k-1)}$ denote the vector containing the nodal values
of all $N+1$ modal components at the $(k-1)$st Picard iteration. To
compute the next iterate, we freeze the first argument of
$\mathcal N(\x;U^{(k-1)},V)$ at $U^{(k-1)}$ and solve the resulting
linear problem with respect to $V$.

We use the Carleman weight
$\varphi_*(x,y)=x+0.5y$ from
Example~\ref{ex:carleman-weight}, which satisfies
\eqref{eq:carleman-weight-condition} with $\mu_0=0.9$.

Let $\mathbf A_H$ denote the fixed sparse matrix obtained by
discretizing $H\cdot\nabla$ using first-order upwind finite
differences. The time integrals appearing in
$\mathcal N(\x;U^{(k-1)},V)$ are evaluated by the trapezoidal rule.
Thus, the discrete approximation of the frozen interior equation is
obtained from
\[
  \mathbf A_H\mathbf V
  +
  \mathcal N(\x;U^{(k-1)},V)
  =
  \mathbf 0,
\]
where $\mathcal N(\x;U^{(k-1)},V)$ is understood as being evaluated
componentwise at all interior grid points. For fixed $U^{(k-1)}$, this
expression is affine with respect to the vector $\mathbf V$ of nodal
values of $V$.

Let $\mathbf B_-$ and $\mathbf B_+$ denote the matrices that extract
the nodal traces on $\Gamma_-$ and $\Gamma_+$, respectively, and let
$\mathbf G$ and $\mathbf D^{\mathrm{obs}}$ contain the corresponding
projected boundary data. We denote by $\mathbf W_\Omega$,
$\mathbf W_-$, and $\mathbf W_+$ the diagonal matrices containing the
Carleman and quadrature weights in the interior and on the two boundary
portions. The regularization term is discretized using a
finite-difference matrix $\mathbf R$ approximating the
$[H^4(\Omega)]^{N+1}$ norm, including all mixed spatial derivatives of
total order at most four.

At the $(k-1)$st Picard iteration, the new modal vector
$\mathbf U^{(k)}$ is obtained by minimizing
\begin{multline}
  \left\|
    \mathbf W_\Omega
    \left[
      \mathbf A_H\mathbf V
      +
      \mathcal N(\x;U^{(k-1)},V)
    \right]
  \right\|_2^2
  +
  \lambda^2
  \left\|
    \mathbf W_-
    \left(
      \mathbf B_-\mathbf V-\mathbf G
    \right)
  \right\|_2^2
  \\
  +
  \lambda^2
  \left\|
    \mathbf W_+
    \left(
      \mathbf B_+\mathbf V-\mathbf D^{\mathrm{obs}}
    \right)
  \right\|_2^2
  +
  \varepsilon
  \|\mathbf R\mathbf V\|_2^2.
  \label{eq:discrete-carleman-functional}
\end{multline}
Because the frozen interior equation is affine with respect to
$\mathbf V$, the functional in
\eqref{eq:discrete-carleman-functional} is quadratic. We compute its
minimizer $\mathbf U^{(k)}$ using MATLAB's built-in sparse linear
least-squares solver. Thus, the interior residual, the two boundary
misfits, and the regularization term are minimized simultaneously at
each Picard iteration.

\begin{Remark}[The admissible ball in theory and computation]
\label{rem:computational-admissible-ball}
The admissible ball $B_M$ plays an important role in the convergence
analysis. It provides a bounded set on which the nonlinear term is
uniformly Lipschitz and on which the Carleman--Picard map is proved to
be contractive. In the numerical implementation, however, we do not
explicitly impose the constraint $\mathbf V\in B_M$. Instead, we solve
the finite-dimensional quadratic least-squares problem without a norm
constraint. This creates no practical restriction because $M$ can be
chosen sufficiently large to contain all computed iterates. With such
a choice, the boundary of $B_M$ is never reached, and the constrained
and unconstrained minimizers coincide. Thus, the ball is needed to formulate and prove the function-space
contraction result, whereas it is inactive in our computation. We emphasize that the constants in the
contraction theorem may depend on $M$; hence, the theoretical estimate
is not asserted to be uniform as $M\to\infty$.
\end{Remark}

In computation, we take
\[
  \lambda=4,
  \qquad
  \varepsilon=10^{-6}.
\]
These parameters are selected manually using only Test~1 and are then
kept fixed for all subsequent tests; no test-dependent retuning is
performed. The Picard iteration is terminated at the first index $k$
satisfying
\begin{equation}
  \frac{
    \|\mathbf U^{(k)}-\mathbf U^{(k-1)}\|_{\ell^2}
  }{
    \max\{\|\mathbf U^{(k)}\|_{\ell^2},10^{-14}\}
  }
  <\tau,
  \label{eq:stopping}
\end{equation}
where $\tau=10^{-4}$. If this criterion is not satisfied earlier, the
iteration is terminated after $K_{\max}=10$ steps.

\subsection{Test profiles}
We first examine a disk inclusion. We report the true initial
state, the reconstructions obtained from $5\%$ and $10\%$ noisy data,
the corresponding curves used to select $N$, and the relative changes
between consecutive Picard iterates. 

\paragraph{Test 1: Disk inclusion.} In this test, the true initial function is given by
\begin{equation}
  u^0(x,y)
  =
  \begin{cases}
    1,&(x-0.30)^2+(y-0.30)^2\leq0.15^2,\\
    0,&\text{otherwise},
  \end{cases}
  \qquad (x,y)\in\Omega.
  \label{eq:disk-profile}
\end{equation}
The reconstruction results are displayed in
Figure~\ref{fig:disk-profile}.
\begin{figure}[!ht]
\centering
\subfloat[True $u^0$]{%
  \includegraphics[width=0.32\textwidth]{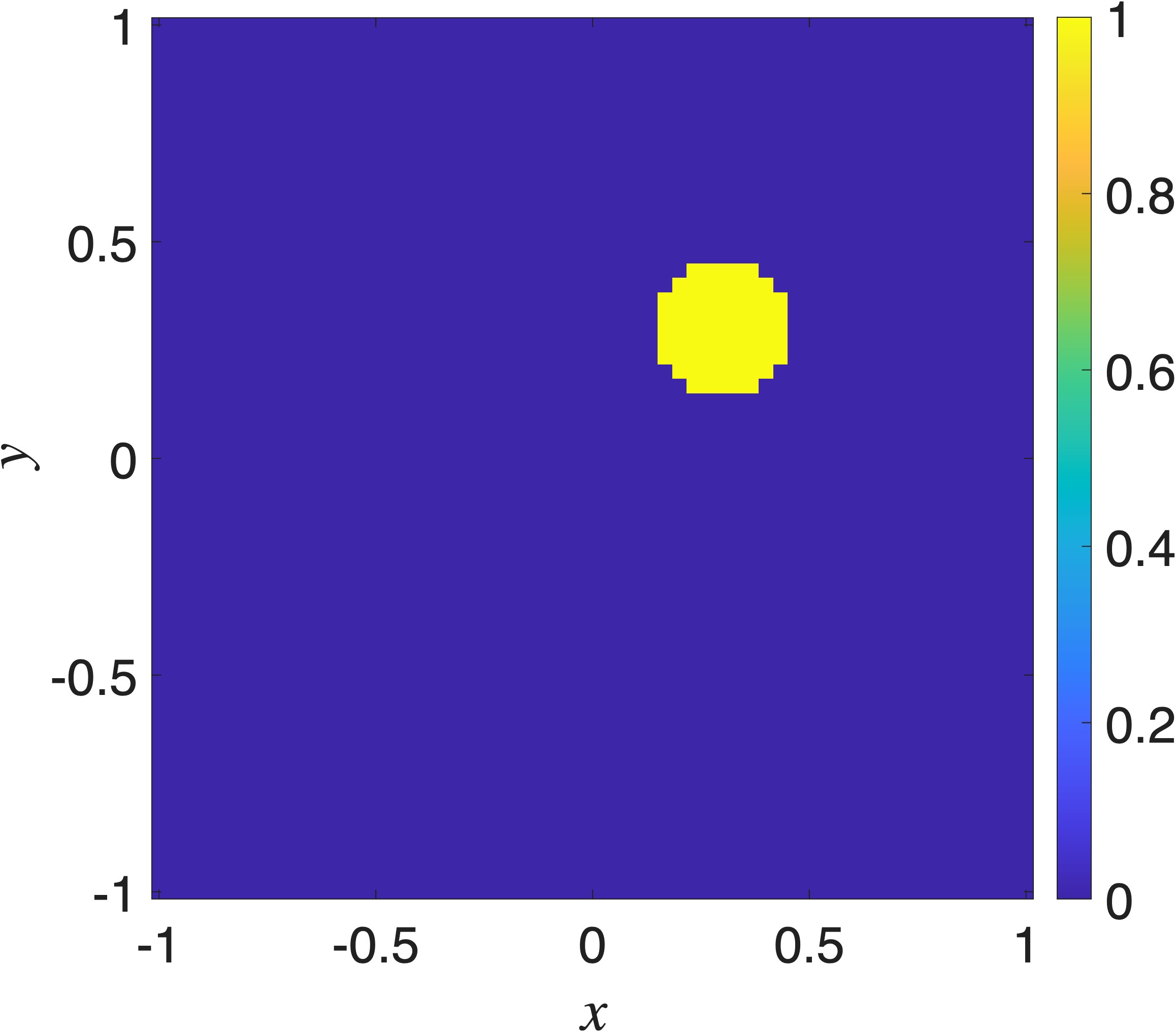}}
\hfill
\subfloat[Reconstruction, $\delta=5\%$]{%
  \includegraphics[width=0.32\textwidth]{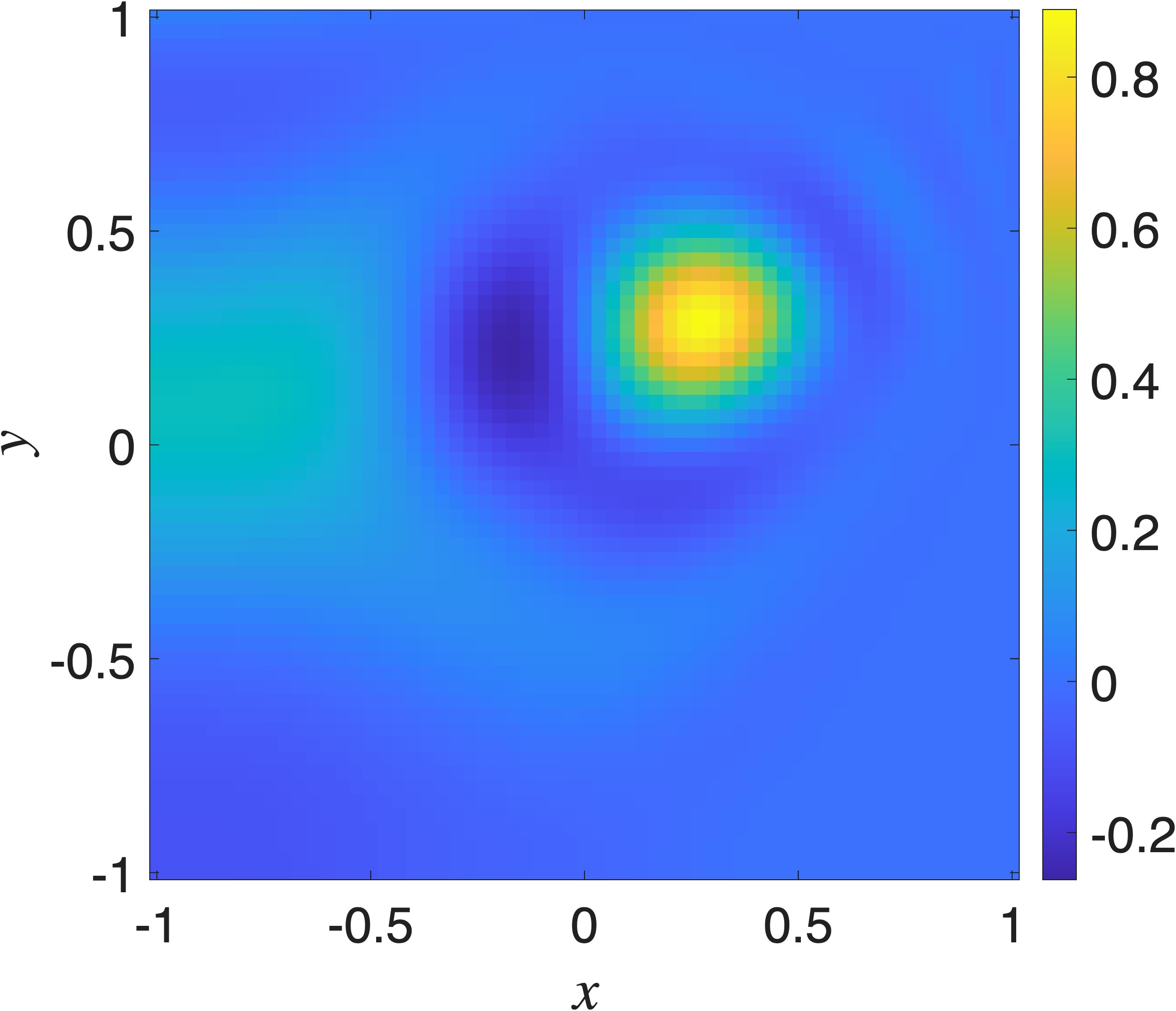}}
\hfill
\subfloat[Reconstruction, $\delta=10\%$]{%
  \includegraphics[width=0.32\textwidth]{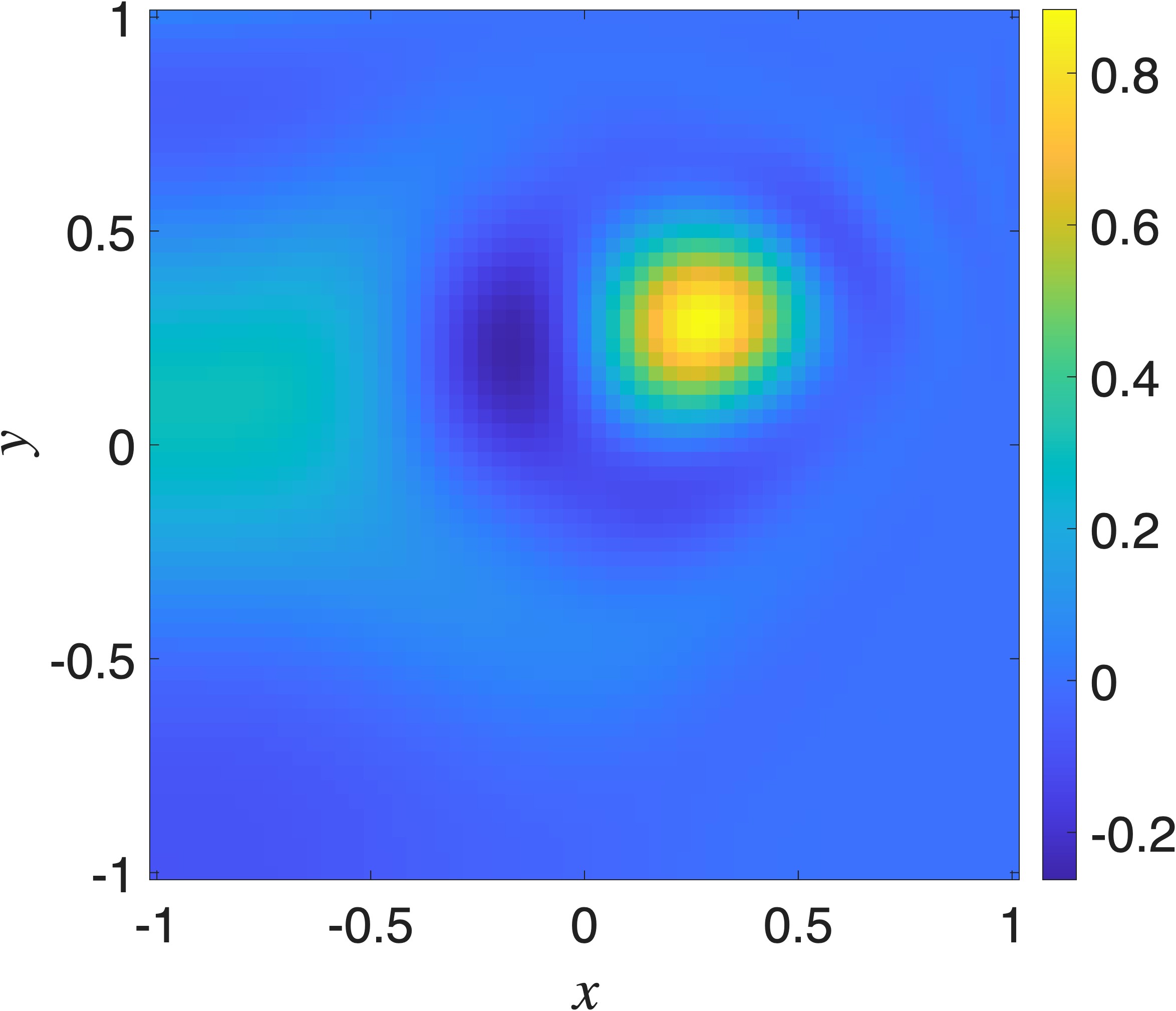}}

\vspace{2pt}
\subfloat[Selection of $N$: $N=13$ for $\delta=5\%$ and
$N=13$ for $\delta=10\%$]{%
  \includegraphics[width=0.48\textwidth]{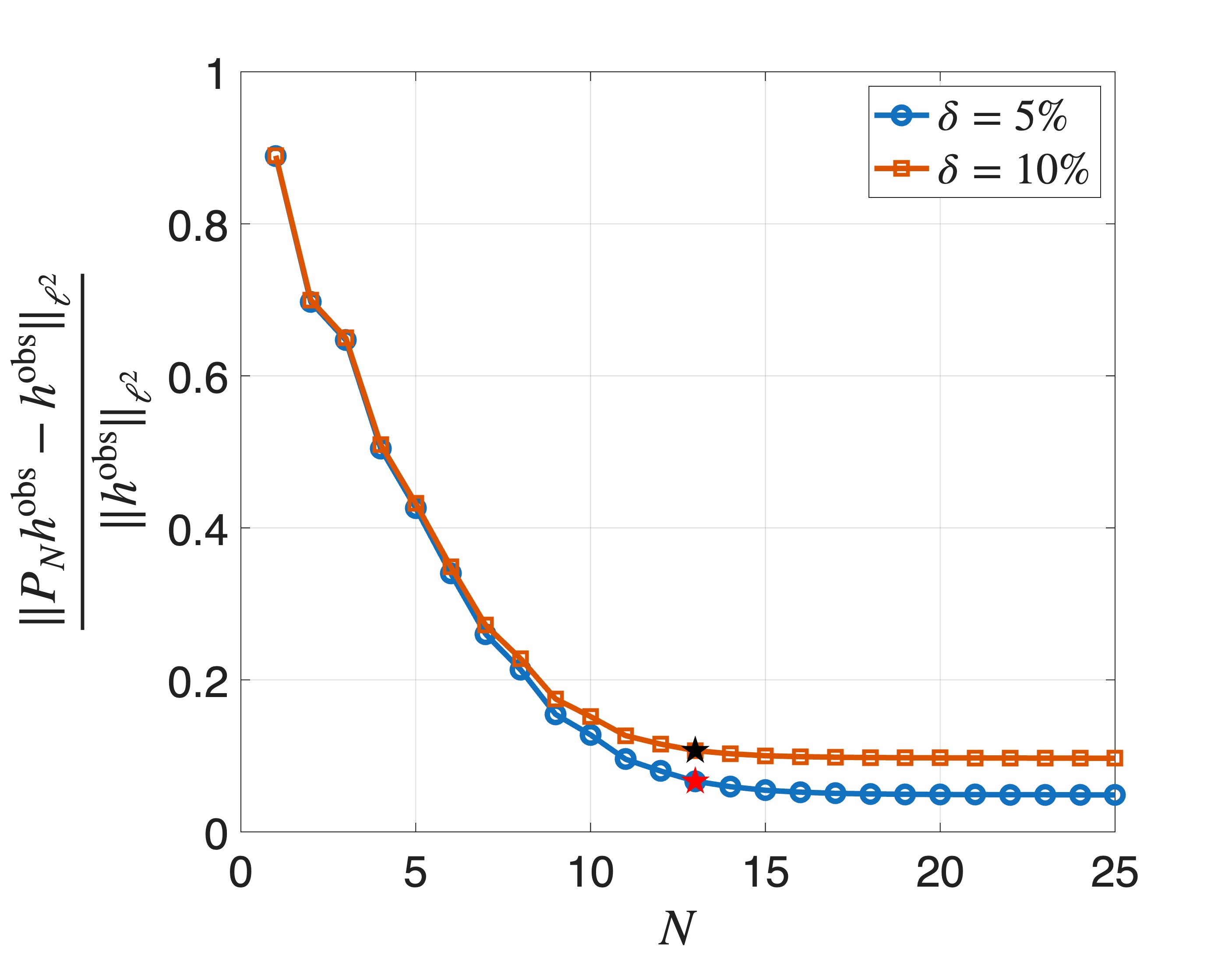}}
\hfill
\subfloat[Relative Picard changes: $k_*=5$ for both noise levels]{%
  \includegraphics[width=0.48\textwidth]{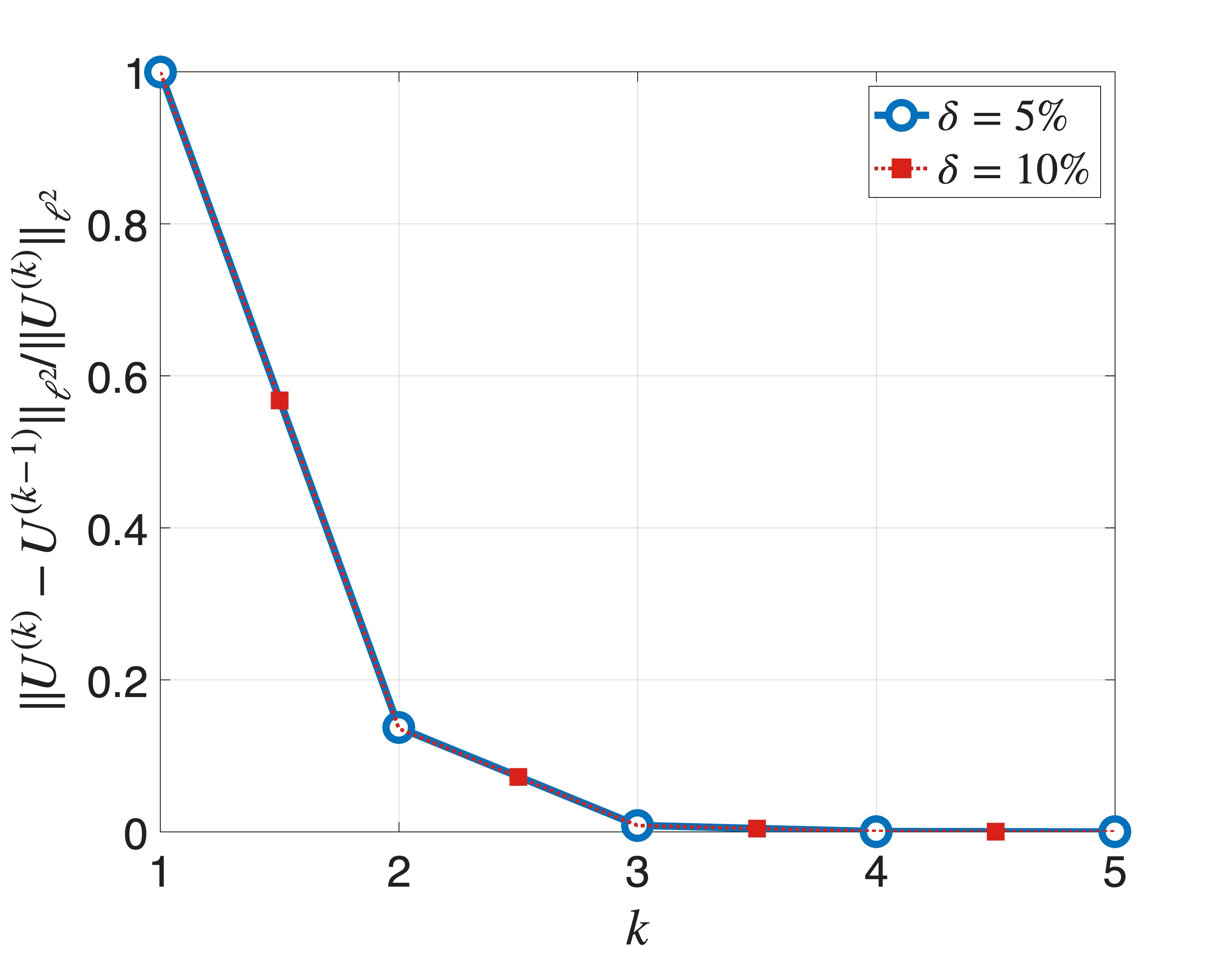}}
\caption{Test 1. The true disk inclusion, the reconstructions, the
selection of $N$, and the relative Picard changes for $5\%$ and $10\%$
noise.}
\label{fig:disk-profile}
\end{figure}

For both noise levels, the plateau criterion selects $N=13$, and the
two projection-residual curves exhibit very similar behavior. For
$\delta=5\%$, the maximum reconstructed value is $0.889461$, yielding
$E_{\max}=11.05\%$. For $\delta=10\%$, the maximum reconstructed value
is $0.883682$, with $E_{\max}=11.63\%$. Thus, doubling the noise level
increases the relative maximum-value error by only $0.58$ percentage
points, indicating that the reconstructed amplitude is stable with
respect to the added noise.

As shown in Figure~\ref{fig:disk-profile}, the disk inclusion is clearly
localized in both reconstructions, and its center remains close to the
true location $(0.30,0.30)$. The principal shape and approximate
spatial extent of the inclusion are also recovered. Because the true
initial condition is discontinuous, the sharp interface of the disk is
smoothed in the reconstructions. Small-amplitude artifacts are also
visible in the background, but they do not obscure the location or the
main structure of the inclusion.

For both noise levels, the stopping index is $k_*=5$. The relative
changes between consecutive Picard iterates decrease
rapidly during the first few iterations and become very small by the
fourth iteration. Moreover, the convergence curves corresponding to
the two noise levels are nearly indistinguishable. These observations
demonstrate the rapid convergence of the Carleman--Picard iteration and
show that, for this test, both the reconstruction and the convergence
history are only weakly affected by increasing the noise level from
$5\%$ to $10\%$.

\paragraph{Test 2: Two Gaussian inclusions.}
In this test, the true initial function is given by
\begin{equation}
  u^0(x,y)
  =e^{-15[(x-0.30)^2+(y+0.40)^2]}
  +0.60e^{-18[(x-0.20)^2+(y-0.35)^2]},
  \qquad (x,y)\in\Omega,
  \label{eq:two-gaussian-profile}
\end{equation}
where the right-hand side is normalized so that its maximum value is
$1$. The reconstruction results are displayed in
Figure~\ref{fig:two-gaussian-profile}.

\begin{figure}[!ht]
\centering
\subfloat[True $u^0$]{%
  \includegraphics[width=0.32\textwidth]{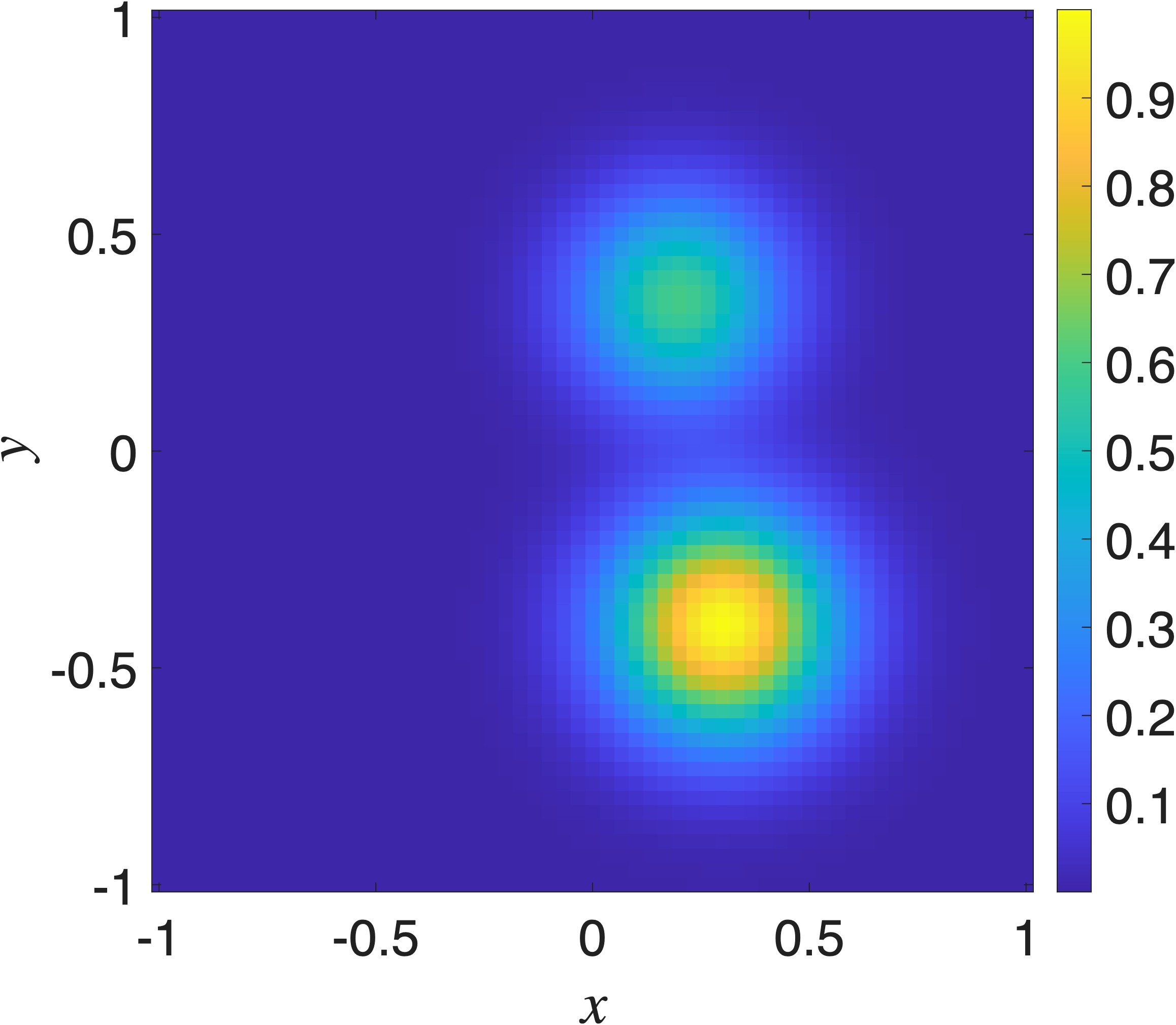}}
\hfill
\subfloat[Reconstruction, $\delta=5\%$]{%
  \includegraphics[width=0.32\textwidth]{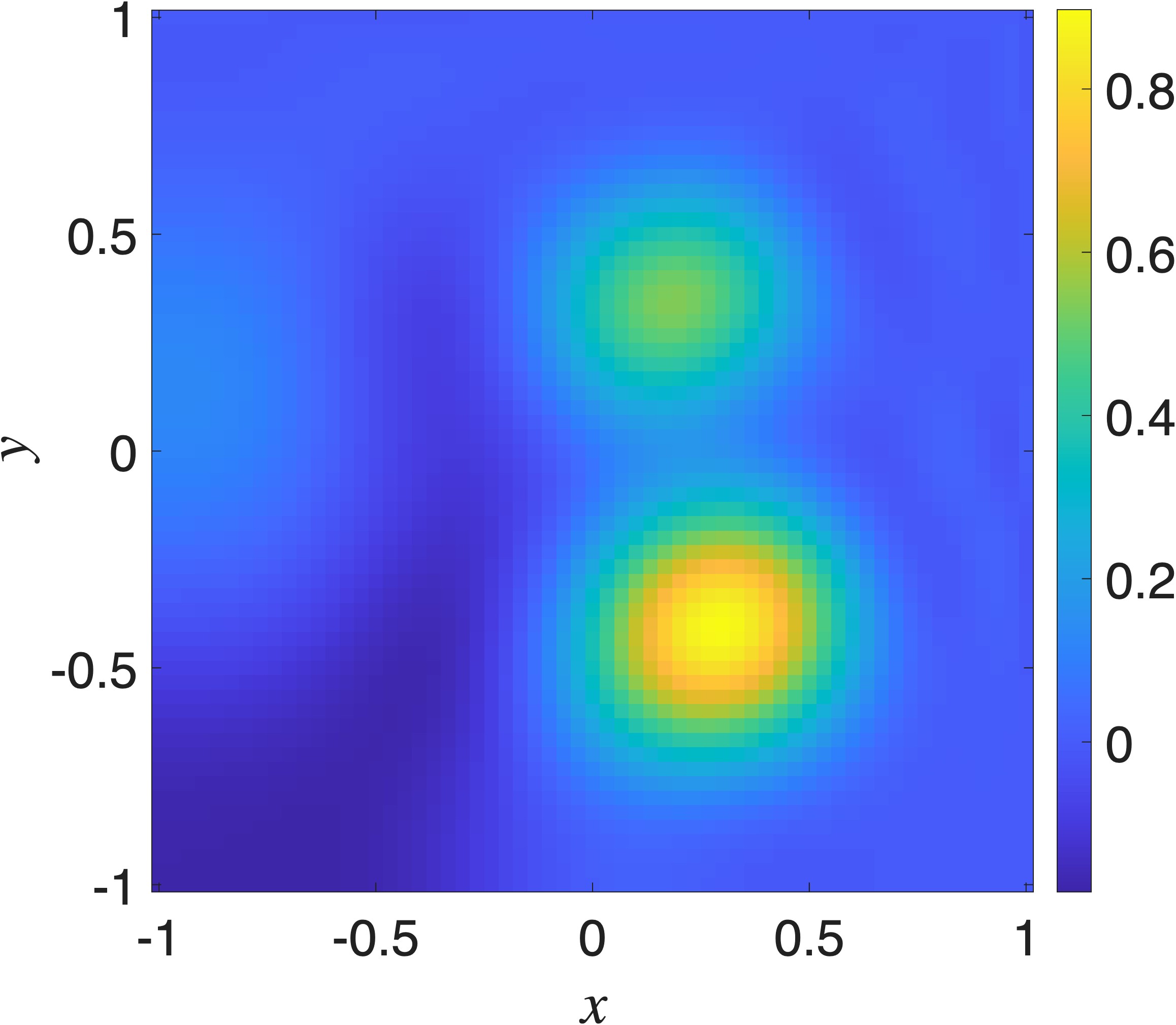}}
\hfill
\subfloat[Reconstruction, $\delta=10\%$]{%
  \includegraphics[width=0.32\textwidth]{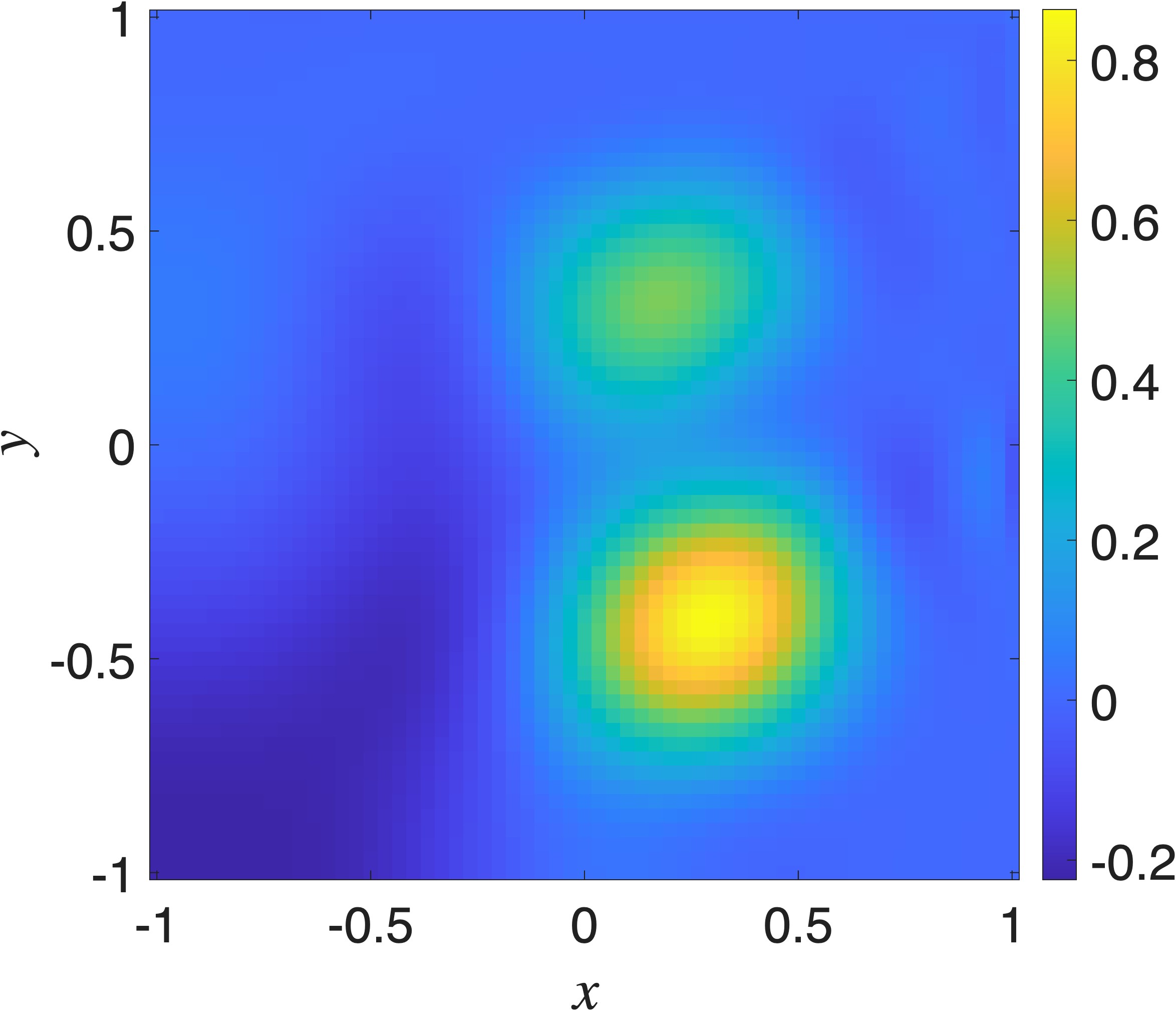}}

\vspace{2pt}
\subfloat[Selection of $N$: $N=10$ for $\delta=5\%$ and
$N=8$ for $\delta=10\%$]{%
  \includegraphics[width=0.48\textwidth]{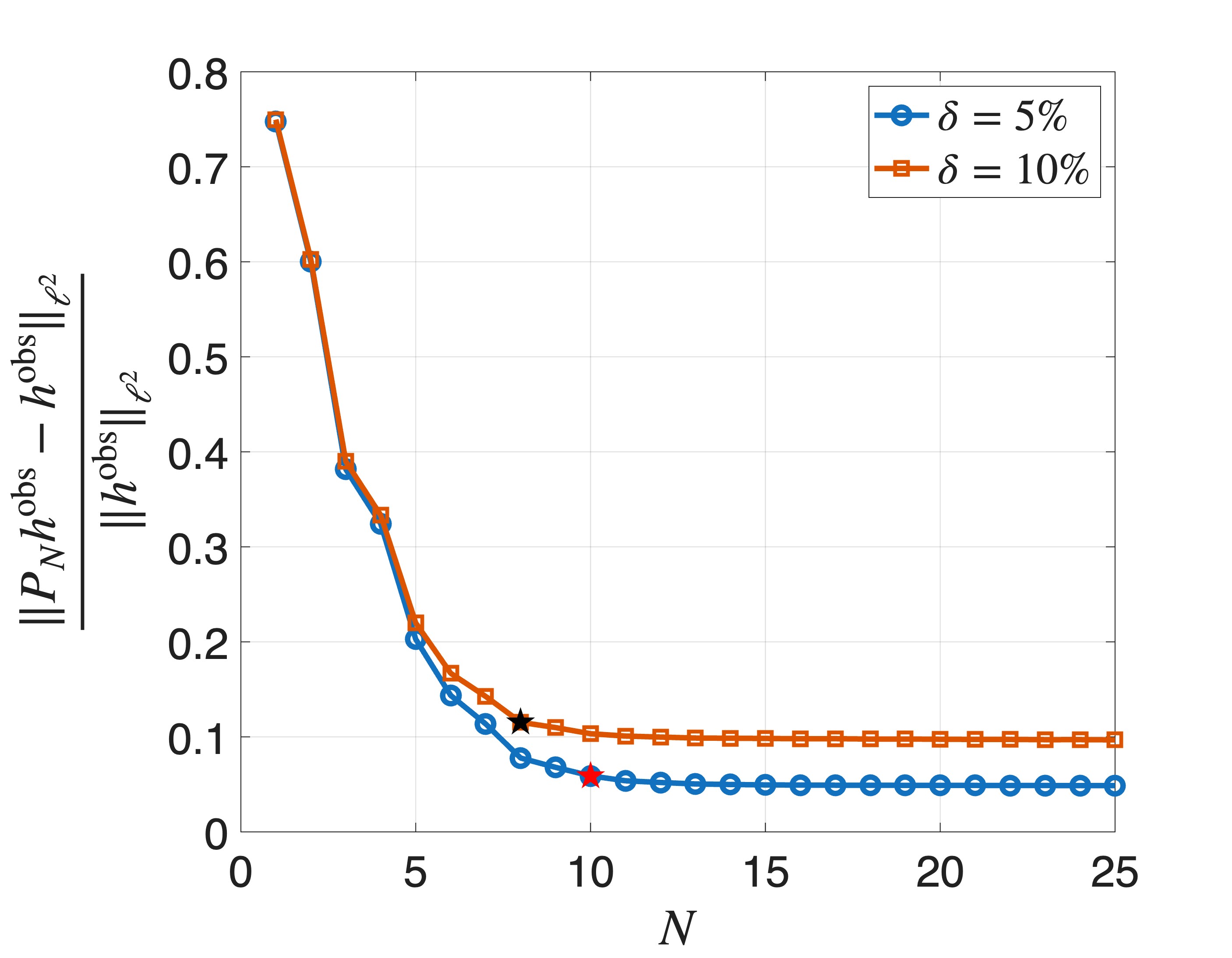}}
\hfill
\subfloat[Relative Picard changes: $k_*=5$ for $\delta=5\%$ and
$k_*=10$ for $\delta=10\%$]{%
  \includegraphics[width=0.48\textwidth]{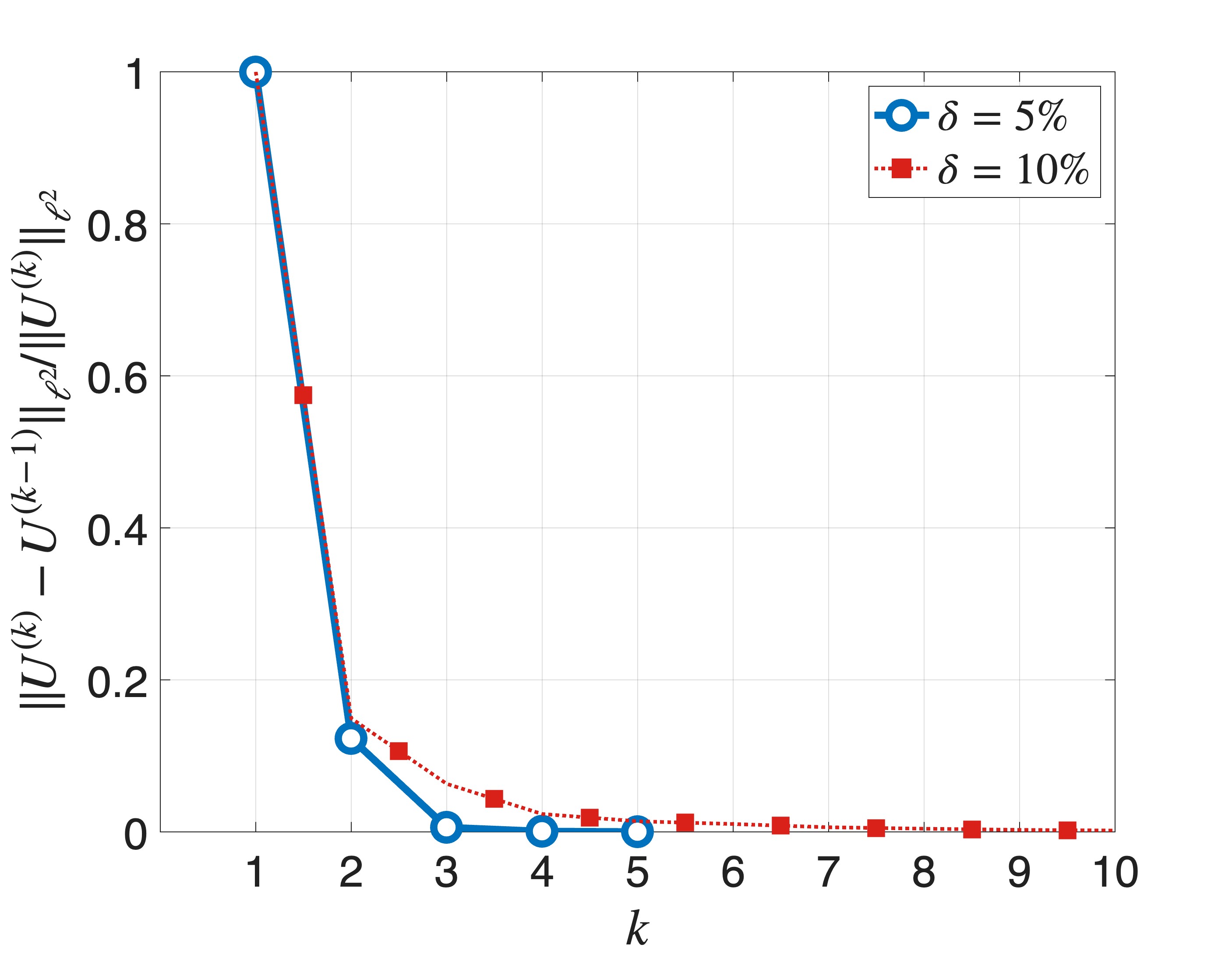}}
\caption{Test 2. The true two-Gaussian profile, the reconstructions, the
selection of $N$, and the relative Picard changes for $5\%$ and $10\%$
noise.}
\label{fig:two-gaussian-profile}
\end{figure}

The plateau criterion selects $N=10$ for $5\%$ noise and $N=8$ for
$10\%$ noise. For $5\%$ noisy data, the reconstructed maxima of the
lower and upper inclusions are $0.897240$ and $0.536076$, respectively,
yielding relative maximum-value errors of $10.28\%$ and $10.24\%$. For
$10\%$ noisy data, the corresponding reconstructed maxima are
$0.863627$ and $0.497659$, with relative errors of $13.64\%$ and
$16.68\%$. Thus, when the noise level increases from $5\%$ to $10\%$,
the relative error increases by $3.36$ percentage points for the lower
inclusion and by $6.44$ percentage points for the upper inclusion. The
weaker upper inclusion is therefore more sensitive to the increased
noise, although its amplitude remains clearly distinguishable from the
background.

As shown in Figure~\ref{fig:two-gaussian-profile}, both inclusions are
clearly visible in the reconstructions. Their centers remain close to
the true locations $(0.30,-0.40)$ and $(0.20,0.35)$, and their spatial
separation is preserved. The reconstruction also retains the correct
ordering of their amplitudes: the lower inclusion remains stronger than
the upper inclusion. Some smoothing and low-amplitude background
artifacts are present, but they do not obscure either peak or cause a
noticeable displacement of the inclusions.

For $5\%$ noise, the relative change between consecutive Picard
iterates decreases sharply and is already close to zero after the third
iteration; the stopping index is $k_*=5$. For $10\%$ noise, the decrease
is slower but remains steady, and the iteration reaches the prescribed
maximum $k_*=K_{\max}=10$. Hence, the Carleman--Picard iteration
converges rapidly for both noise levels, while the noisier data require
more iterations and do not attain the stopping tolerance before
$K_{\max}$.

\paragraph{Test 3: Y-shaped inclusion.}
This test contains a single inclusion in the shape of the letter Y. The
true initial function is
\begin{equation}
  u^0(x,y)=\begin{cases}
    1,&(x,y)\in\mathcal Y,\\
    0,&(x,y)\in\Omega\setminus\mathcal Y,
  \end{cases}
  \label{eq:Y-profile}
\end{equation}
where $\mathcal Y$ is formed by two upper branches joining
$(0.32,0.90)$ and $(0.84,0.90)$ to $(0.58,0.58)$, together with the
vertical stem joining $(0.58,0.58)$ to $(0.58,0.22)$. Each stroke has
half-width $0.055$. The reconstruction results are displayed in
Figure~\ref{fig:Y-profile}.

\begin{figure}[!ht]
\centering
\subfloat[True $u^0$]{%
  \includegraphics[width=0.32\textwidth]{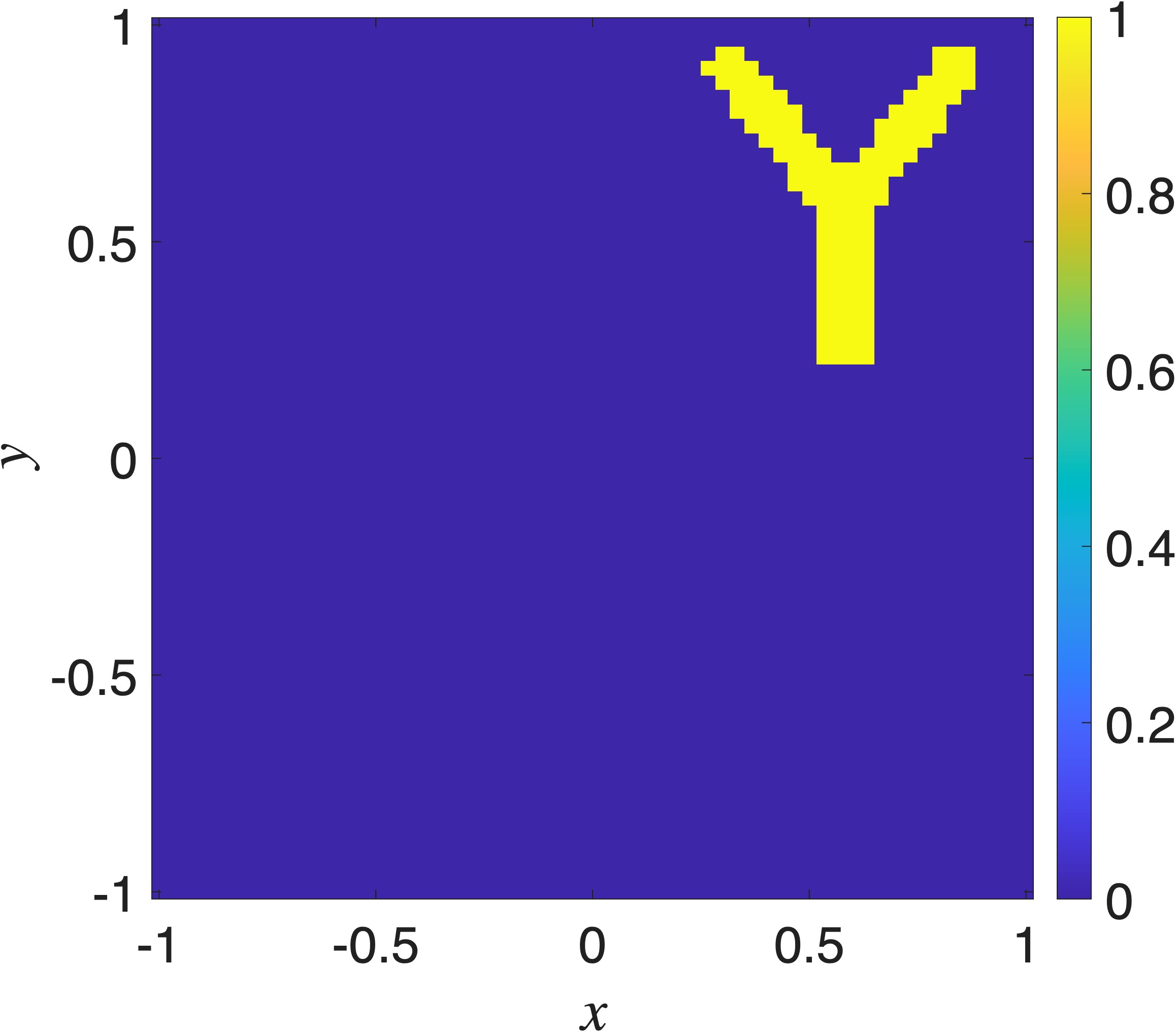}}
\hfill
\subfloat[Reconstruction, $\delta=5\%$]{%
  \includegraphics[width=0.32\textwidth]{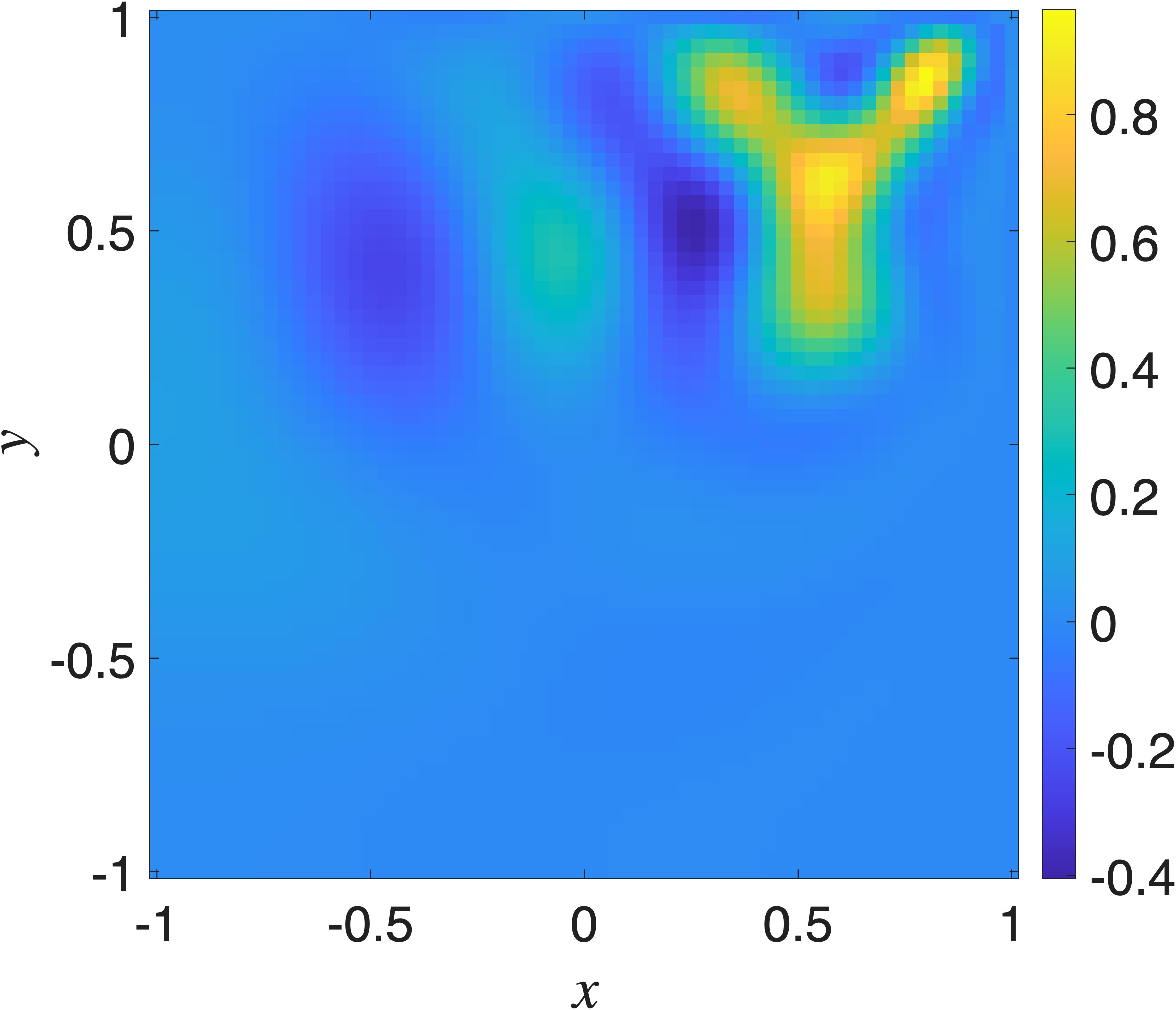}}
\hfill
\subfloat[Reconstruction, $\delta=10\%$]{%
  \includegraphics[width=0.32\textwidth]{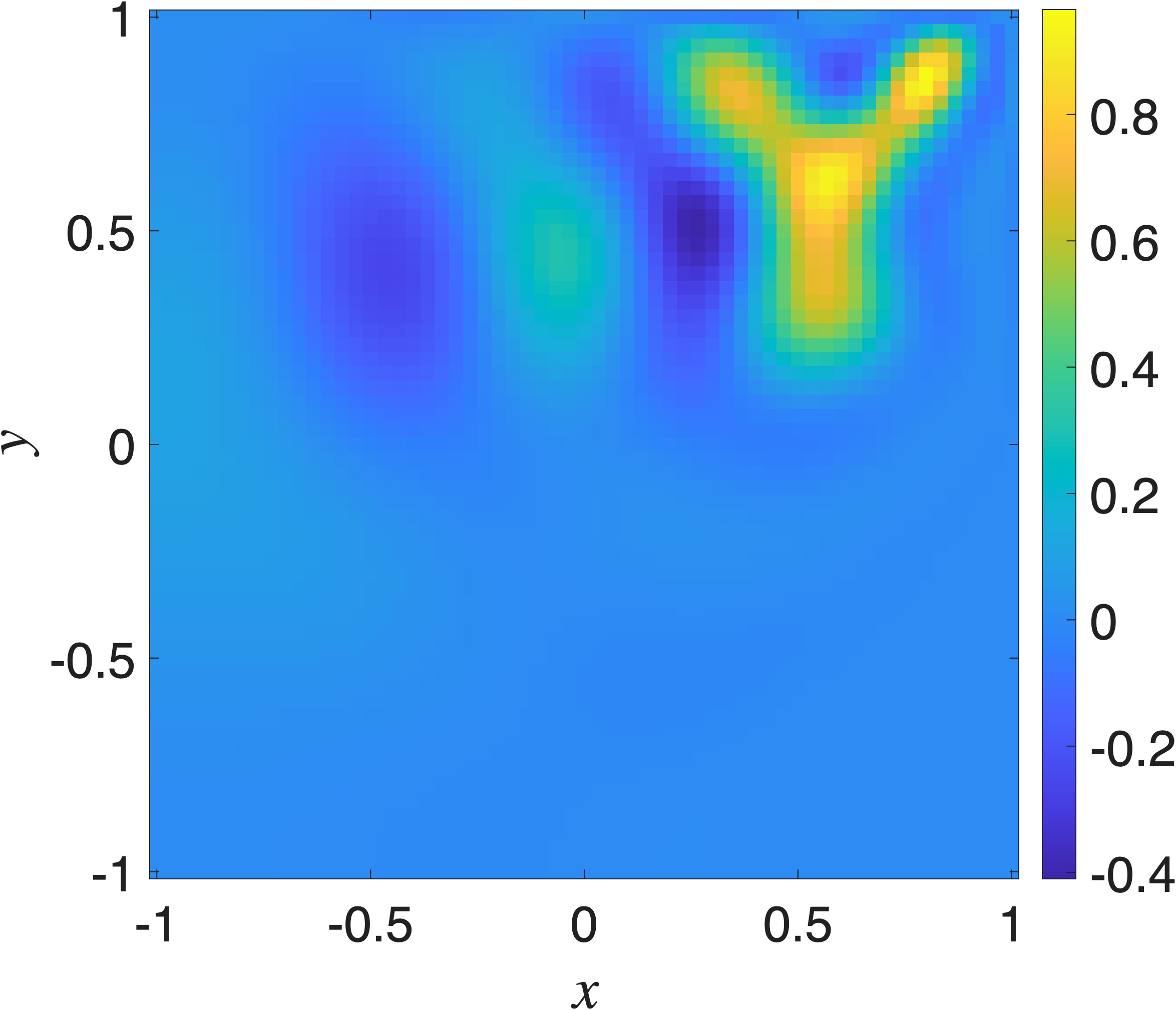}}

\vspace{2pt}
\subfloat[Selection of $N$: $N=14$ for both noise levels]{%
  \includegraphics[width=0.48\textwidth]{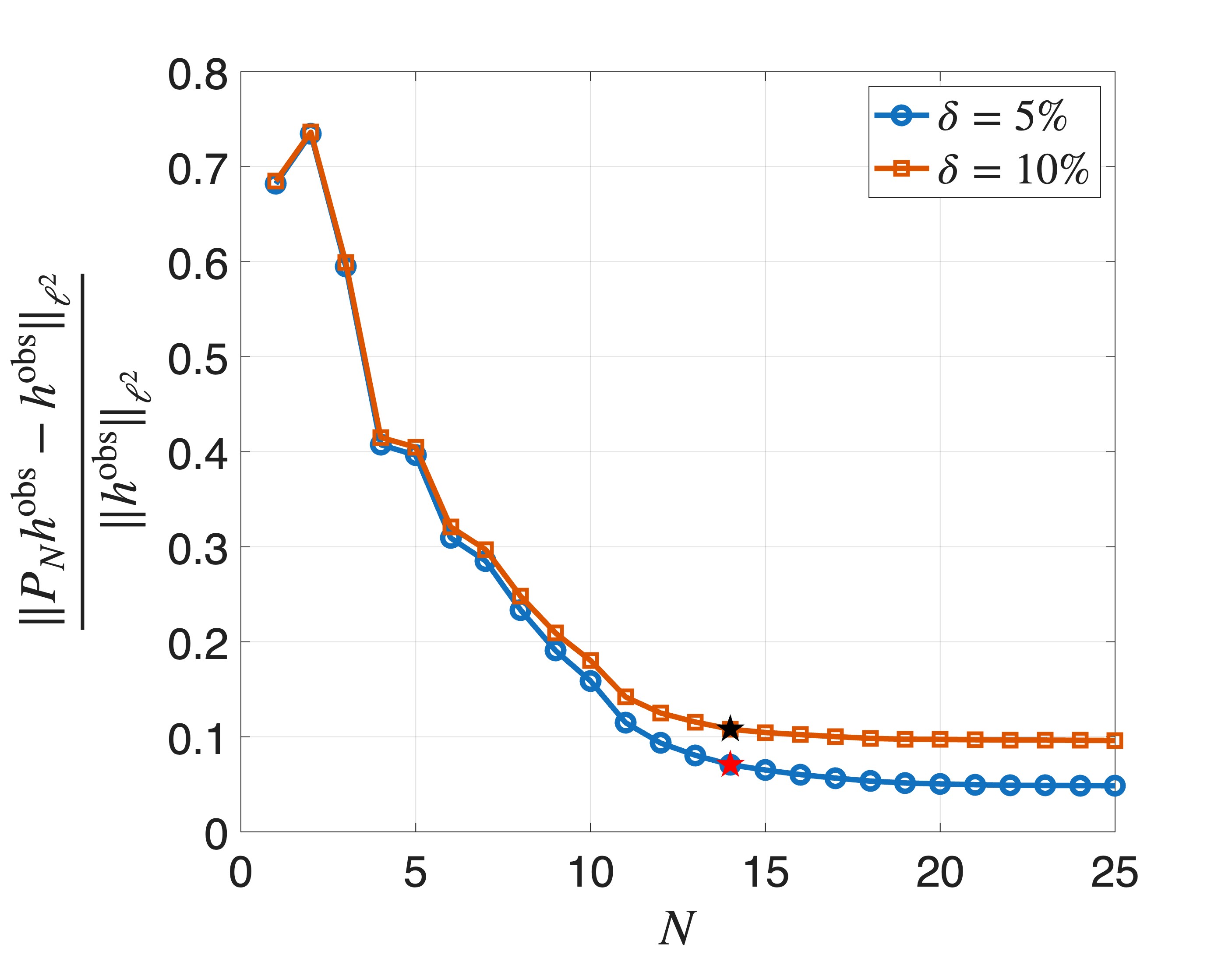}}
\hfill
\subfloat[Relative Picard changes: $k_*=5$ for both noise levels]{%
  \includegraphics[width=0.48\textwidth]{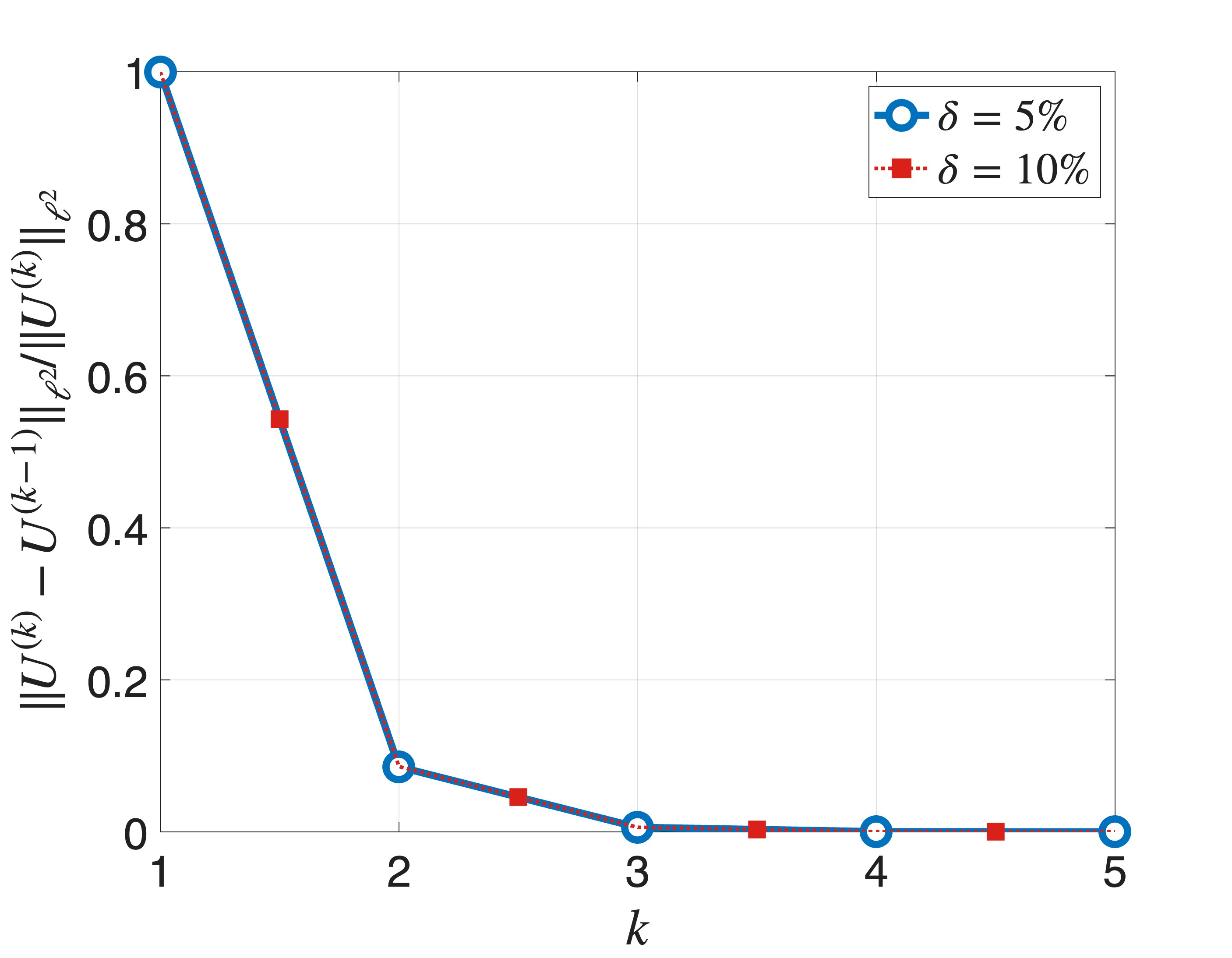}}
\caption{Test 3. The true Y-shaped inclusion, the reconstructions, the
selection of $N$, and the relative Picard changes for $5\%$ and $10\%$
noise.}
\label{fig:Y-profile}
\end{figure}

For both noise levels, the plateau criterion selects $N=14$. For
$5\%$ noisy data, the maximum reconstructed value is $0.965675$, which
corresponds to a relative maximum-value error of $3.43\%$. For $10\%$
noisy data, the maximum reconstructed value is $0.966293$, with a
relative error of $3.37\%$. The difference between the two relative
errors is only $0.06$ percentage points. Thus, for this test, increasing
the noise level from $5\%$ to $10\%$ has essentially no effect on the
recovered maximum amplitude.

As shown in Figure~\ref{fig:Y-profile}, the Y-shaped inclusion remains
clearly recognizable in both reconstructions and is recovered in the
correct upper-right region of $\Omega$. The two upper branches, their
junction, and the vertical stem are all visible, so the main topology
and orientation of the inclusion are preserved. The thin strokes of the
true discontinuous inclusion become wider and smoother in the
reconstructions. Oscillatory background artifacts are also visible,
particularly to the left of the inclusion, but they do not obscure the
Y shape or shift its principal location.

For both noise levels, the stopping index is $k_*=5$. The two relative
Picard-change curves are nearly indistinguishable: they decrease sharply
during the first three iterations and are close to zero thereafter.
This behavior indicates rapid convergence and shows that the convergence
history is stable with respect to the increase in the noise level.

\begin{Remark}[Regularity of the numerical test profiles]
\label{rem:numerical-regularity}
The analytical framework assumes
$u\in C^1(\overline\Omega\times[0,T])$, and therefore requires a
correspondingly smooth initial state. The disk and Y-shaped profiles
used above are characteristic functions with discontinuous interfaces,
so they do not satisfy this regularity assumption. We include these
examples as numerical stress tests to examine whether the reconstruction
method can still identify the location and geometry of sharply defined
inclusions. The associated computations should not be interpreted as a
verification or extension of the convergence and stability theorems to
discontinuous solutions. Such an extension would require a separate
analysis in an appropriate weak-solution framework. By contrast, the
two-Gaussian profile is smooth and is consistent with the regularity
assumed in the theoretical development.
\end{Remark}

\subsection{Influence of the Carleman parameter}
We next examine the influence of the Carleman parameter on Test~2 using
the same $5\%$ noisy data realization. We vary $\lambda$ from $0$ to
$7$, while all remaining computational parameters are kept fixed. For
$\lambda=0$, the exponential weight is absent and the boundary misfit
is retained with unit weight; this case therefore represents the
standard unweighted Picard iteration. Selected reconstructions are
shown in Figure~\ref{fig:lambda-study}.

\begin{figure}[!ht]
\centering
\subfloat[$\lambda=0$]{%
  \includegraphics[width=0.32\textwidth]{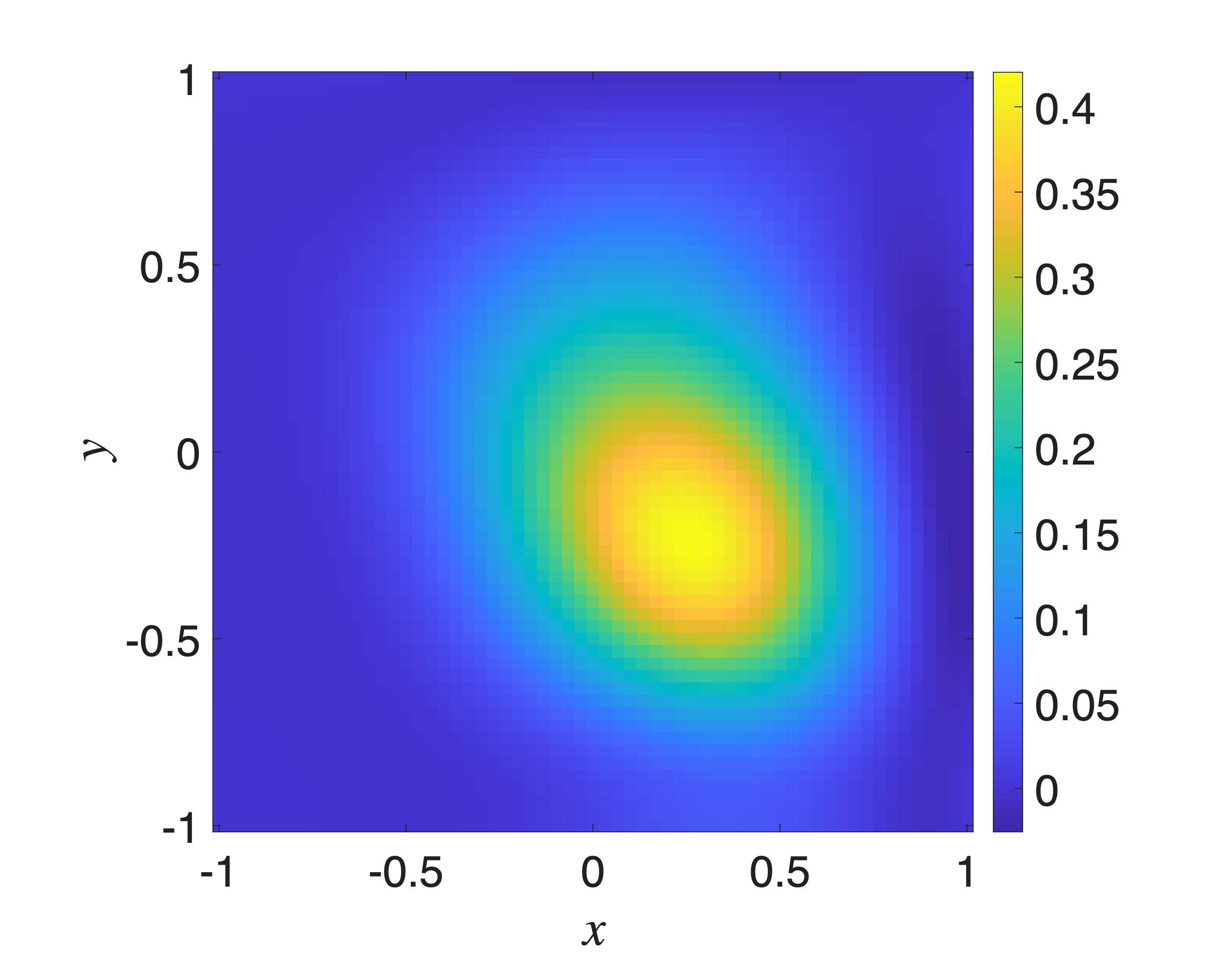}}
\hfill
\subfloat[$\lambda=1$]{%
  \includegraphics[width=0.32\textwidth]{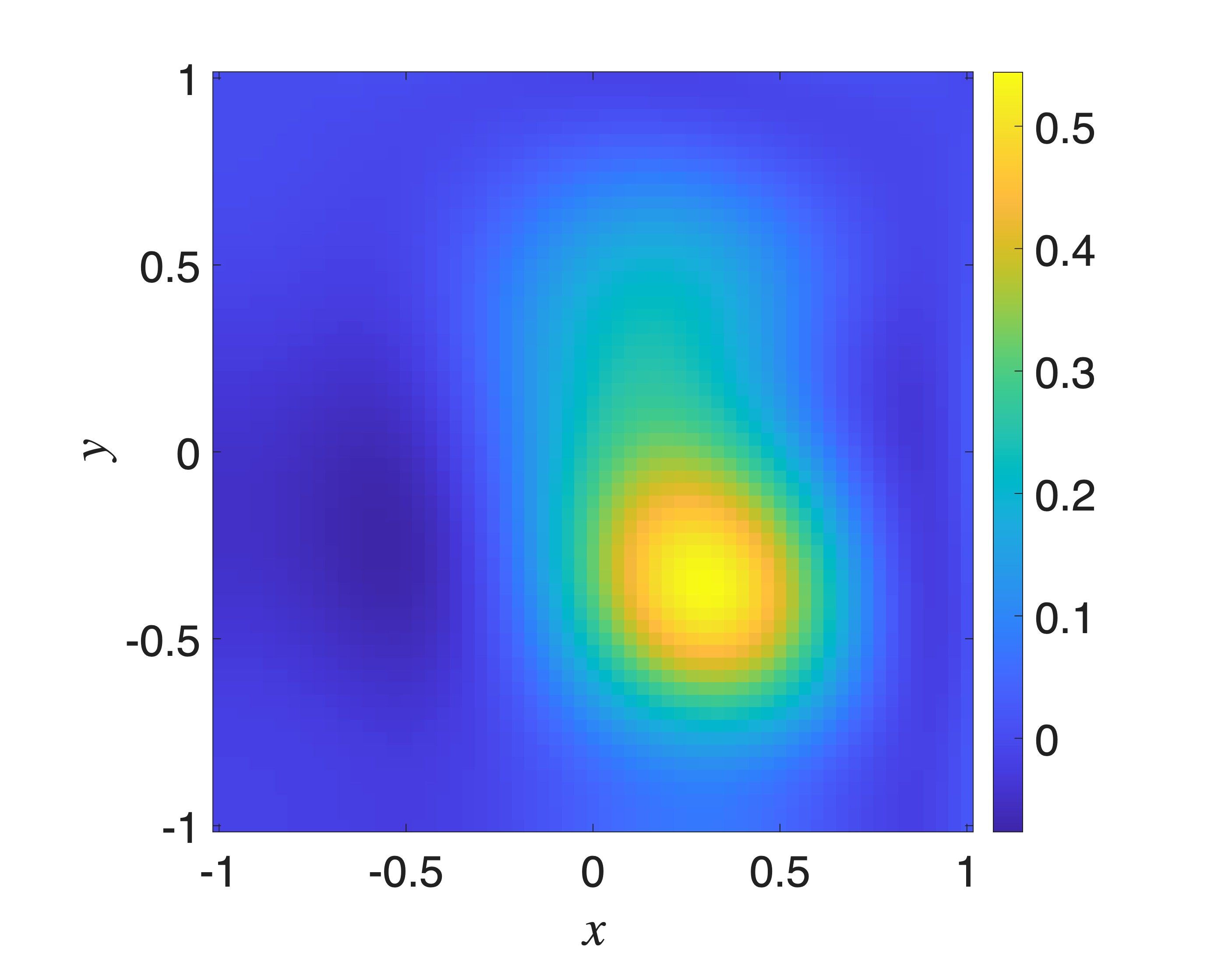}}
\hfill
\subfloat[$\lambda=2$]{%
  \includegraphics[width=0.32\textwidth]{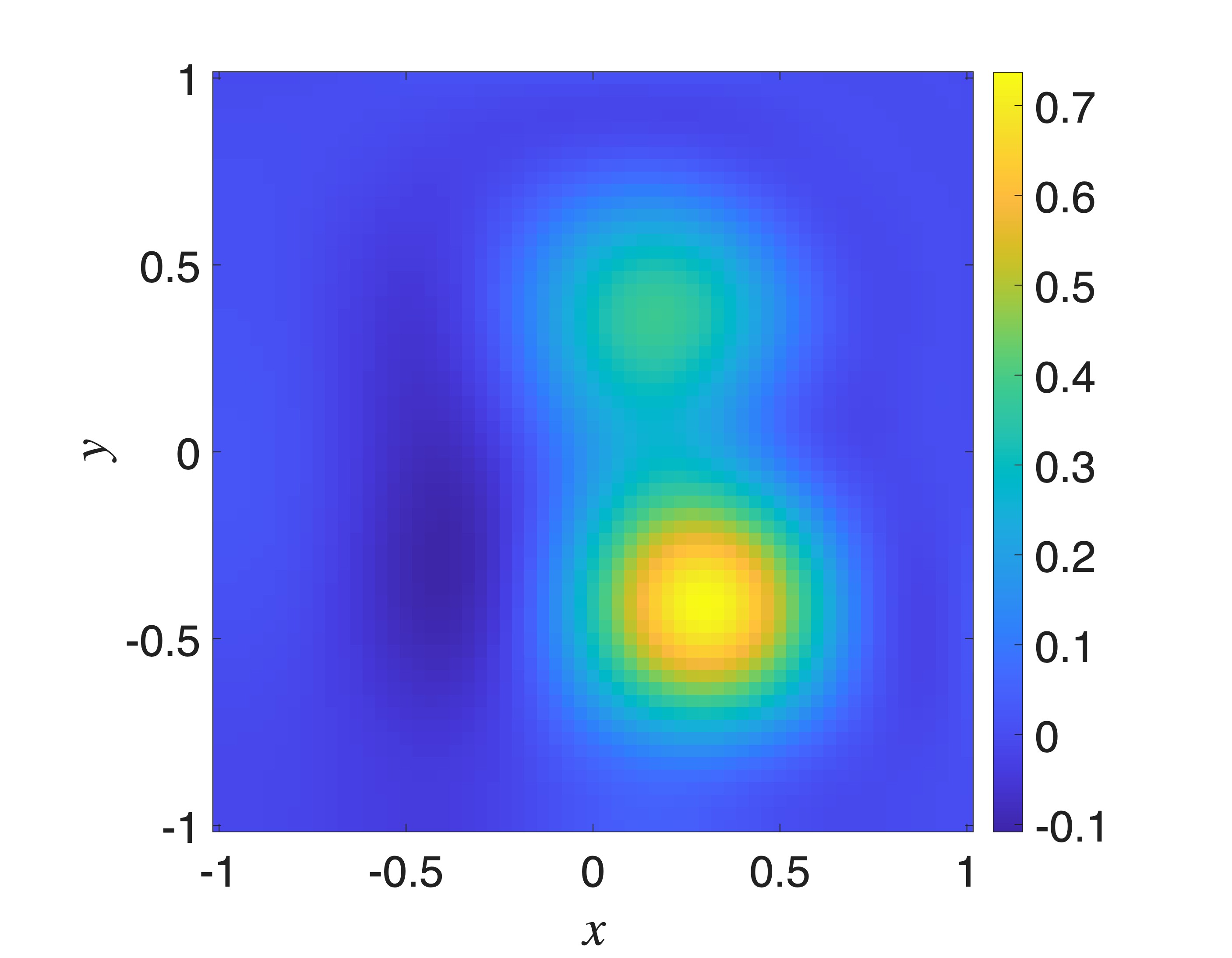}}

\vspace{2pt}
\subfloat[$\lambda=3$]{%
  \includegraphics[width=0.32\textwidth]{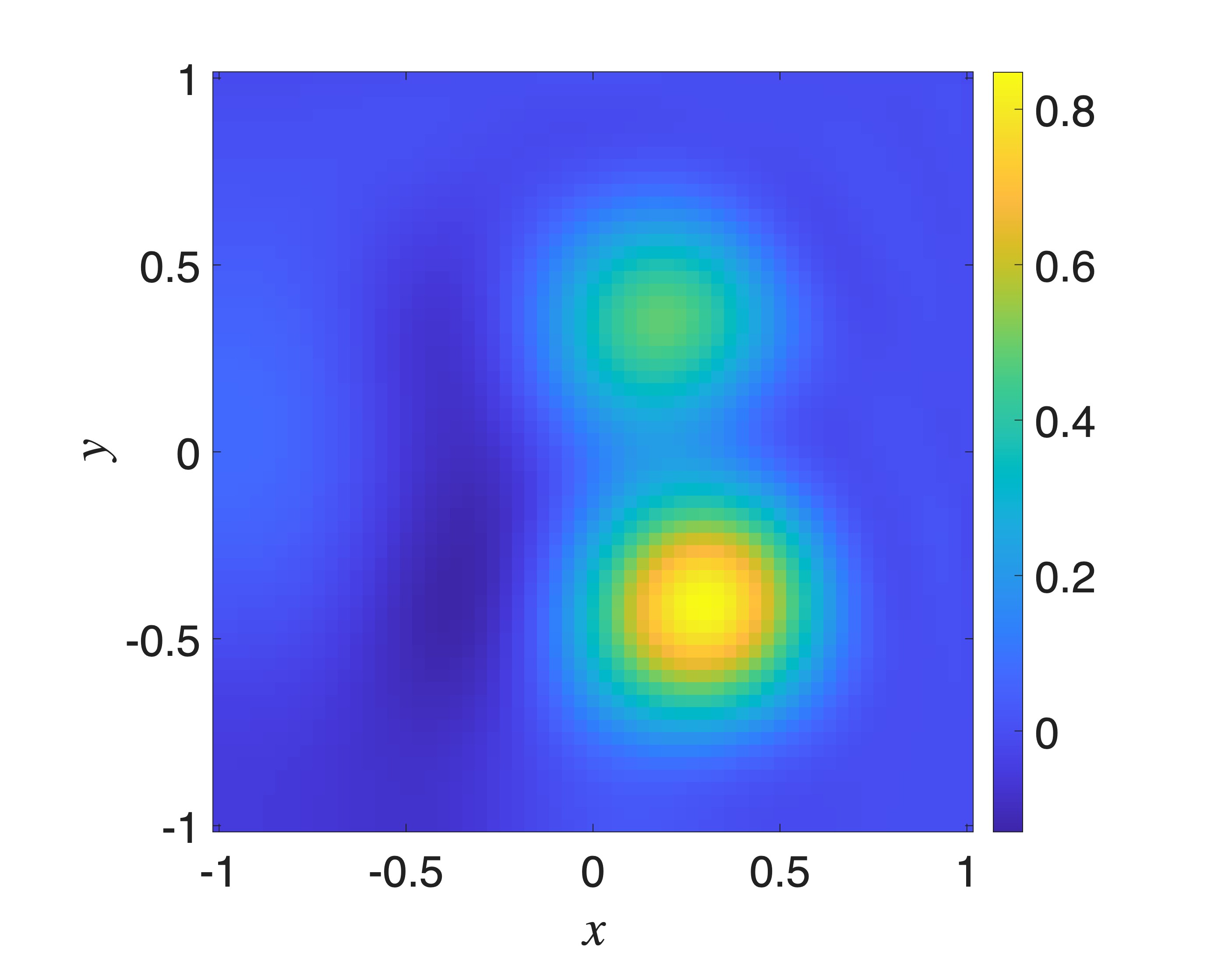}}
\hfill
\subfloat[$\lambda=4$]{%
  \includegraphics[width=0.32\textwidth]{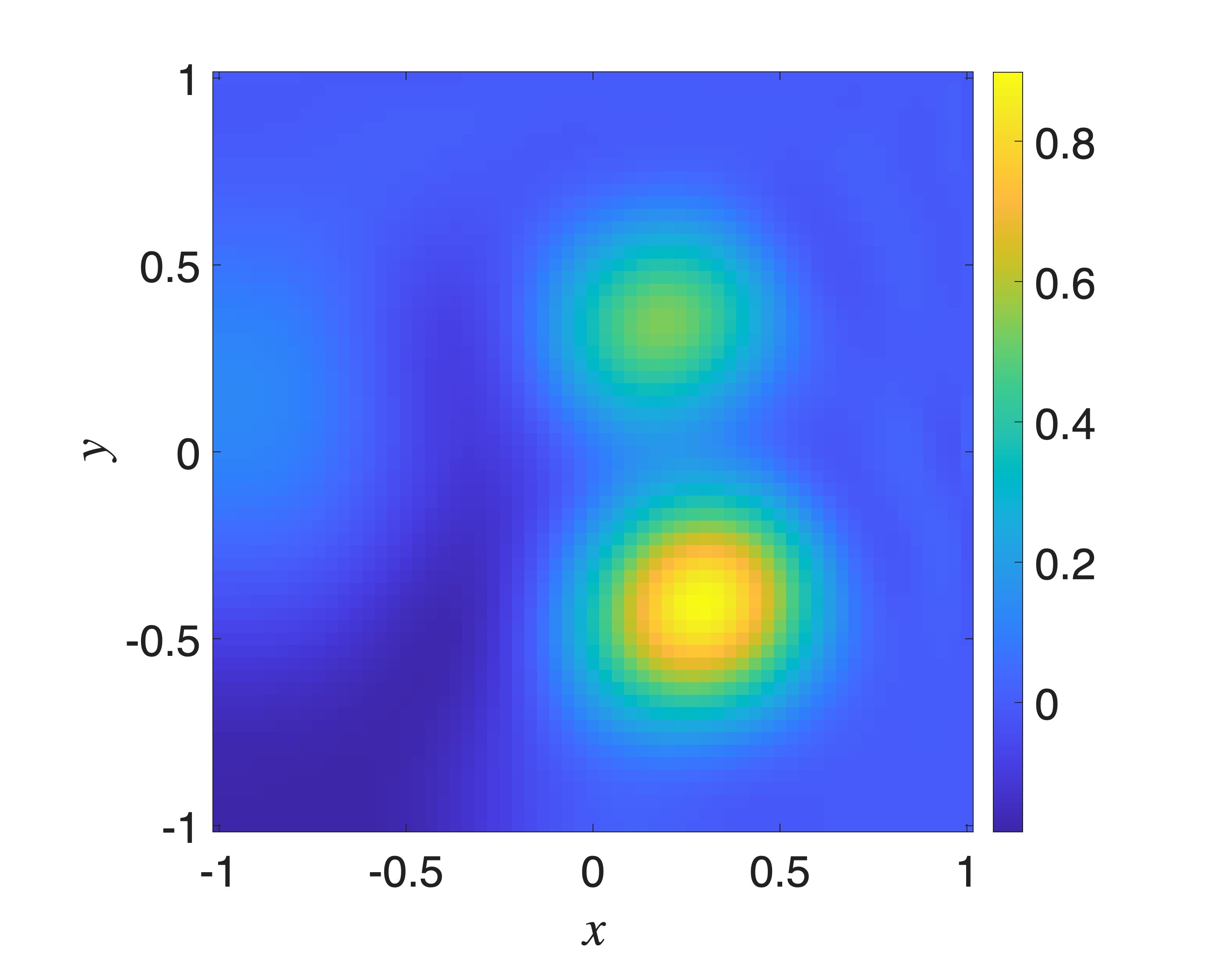}}
\hfill
\subfloat[$\lambda=7$]{%
  \includegraphics[width=0.32\textwidth]{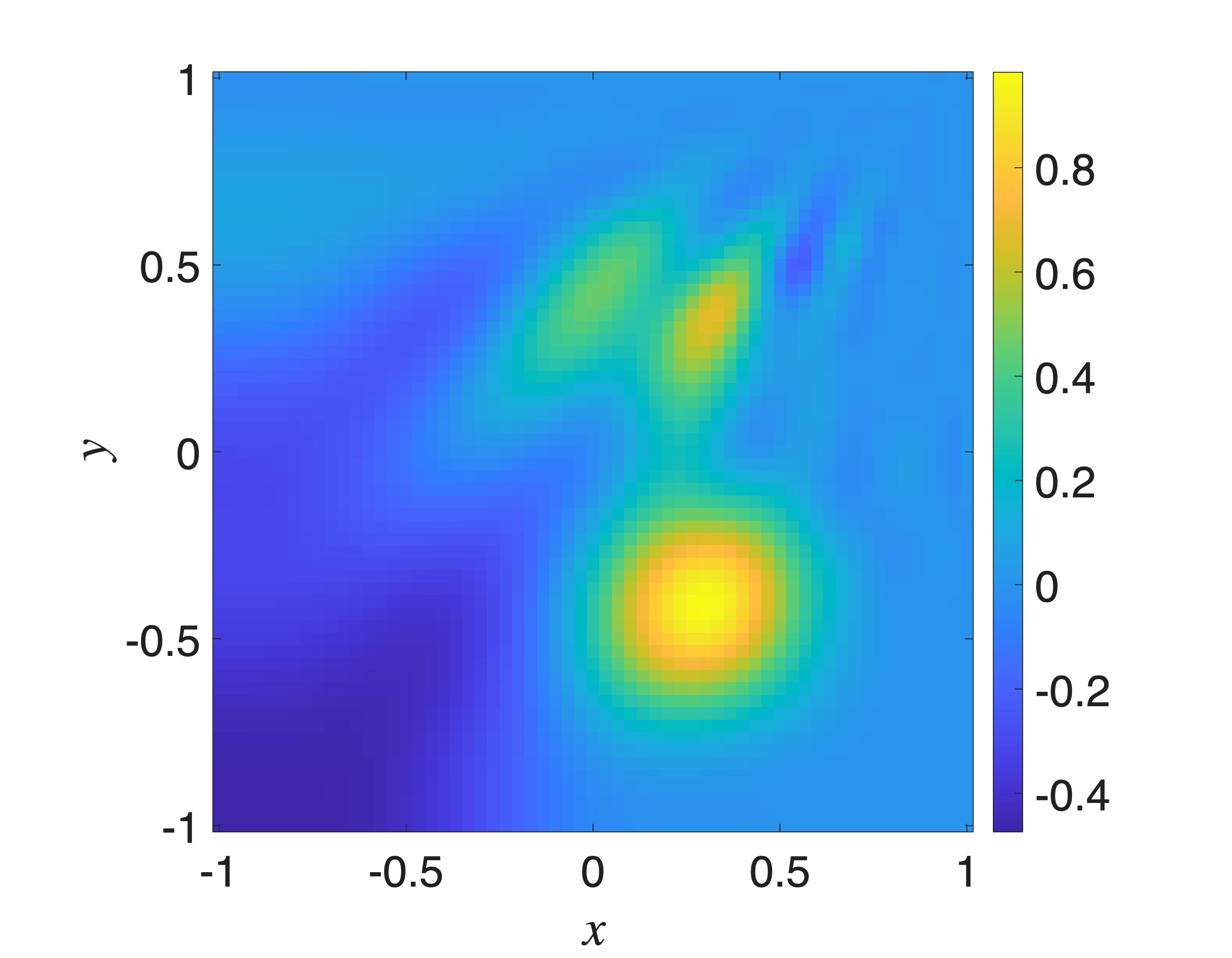}}
\caption{Reconstructions of the two-Gaussian initial state from $5\%$
noisy data for several values of the Carleman parameter $\lambda$.}
\label{fig:lambda-study}
\end{figure}

The reconstruction obtained with $\lambda=0$, for which no Carleman
weight is used, is not satisfactory. In particular, the two inclusions
are not recovered as clearly as in the reconstructions obtained with
positive, moderate values of $\lambda$. The results for
$\lambda=1,2,3,$ and $4$ show that introducing the Carleman weight
improves the localization and separation of the two inclusions. This
comparison illustrates the practical contribution of the Carleman
weight, beyond merely using an ordinary unweighted Picard iteration.

On the other hand, increasing $\lambda$ indefinitely is not beneficial
in numerical computation. The reconstruction for $\lambda=7$
deteriorates because of numerical conditioning; this does not
contradict the continuous Carleman estimate. In the present experiment,
$\varphi_*(x,y)=x+0.5y$ ranges from $-1.5$ to $1.5$ on $\Omega$.
Consequently, when $\lambda=7$, the functional weight
$e^{2\lambda\varphi_*}$ ranges from $e^{-21}$ to $e^{21}$, whose ratio
is $e^{42}$, approximately $1.7\times10^{18}$. This enormous dynamic
range makes the discrete weighted least-squares system severely
ill-conditioned and amplifies finite-precision and discretization
errors. Therefore, although the theory requires $\lambda$ to be
sufficiently large to obtain the contraction property, $\lambda$ must
remain within a moderate computational range. Excessively large values
should be avoided in practice.

\section{Conclusion}
\label{sec:conclusion}
We have developed a Carleman--Picard method for reconstructing the
initial state of a nonlinear transport equation with memory from
outflow boundary observations. The temporal variable was eliminated by
expanding the solution in a Legendre--exponential basis and truncating
the expansion, which led to a finite coupled nonlinear transport system
for the spatial modal coefficients. A Carleman estimate for $H\cdot\nabla$ was used to prove that the
resulting Picard map is contractive for a sufficiently large Carleman
parameter.

Within the truncated and regularized setting, the Carleman-weighted
minimization problem at every iteration has a unique solution. For a
sufficiently large Carleman parameter, the induced Picard map is a
strict contraction on the admissible ball, with contraction factor of
order $\lambda^{-1/2}$. This gives convergence from any initialization
in that ball, rather than only from an initial guess close to the exact
reduced solution. The noisy-data analysis also provides a weighted
stability estimate consisting of a boundary-data error term and a
Tikhonov regularization term.

The numerical experiments support the analytical construction. The
method recovered a disk, two separated Gaussian inclusions, and a
Y-shaped inclusion from multiplicatively perturbed outflow data. The
principal locations and geometries remained visible at both tested
noise levels, and the relative changes between consecutive iterates
typically decreased rapidly. The Carleman-parameter study also
highlighted an important computational balance. The unweighted Picard
scheme produced an unsatisfactory reconstruction, while moderate
positive values of $\lambda$ improved the result. On the other hand,
an excessively large value generated a very large dynamic range in the
exponential weight and severely degraded the conditioning of the
discrete least-squares system.

The theoretical conclusions of this work concern the truncated and
regularized reduced problem for fixed values of the truncation index
and the Tikhonov regularization parameter. The temporal truncation is
not merely a numerical approximation; it also serves as an additional
layer of regularization by suppressing high-order temporal modes that
may amplify noise and instability. The Tikhonov term provides a second
regularization mechanism for the resulting spatial system. Accordingly,
the convergence and stability results established here apply to this
combined truncation--Tikhonov framework.

\section*{CRediT authorship contribution statement}

\textbf{Navaraj Neupane:} Conceptualization, Methodology, Formal analysis, Software, Validation, Visualization, Writing--original draft, Writing--review and editing.

\textbf{Loc Nguyen:} Conceptualization, Methodology, Formal analysis, Software, Validation, Visualization, Writing--original draft, Writing--review and editing.

Both authors read and approved the final manuscript.

\section*{Funding}

This research did not receive any specific grant from funding agencies in the public, commercial, or not-for-profit sectors.

\section*{Declaration of competing interest}

The authors declare that they have no known competing financial interests or personal relationships that could have appeared to influence the work reported in this paper.

\section*{Data availability}

No external datasets were used in this study. The MATLAB code used to generate the numerical results in this paper is available at \url{https://doi.org/10.5281/zenodo.21865108} on Zenodo. Running the provided code reproduces the synthetic outflow boundary data and the reconstruction results reported in the manuscript.

\bibliographystyle{plain}
\bibliography{Transport}
\end{document}